\documentclass[authoryear,preprint,12pt]{elsarticle}

\usepackage{amssymb}
\usepackage{amsthm}

\usepackage{bm,amsmath,algorithm,algorithmic,url,color,booktabs}
\usepackage{comment}
\usepackage{cases}
\usepackage{multirow}
\usepackage{threeparttable}
\usepackage{mathtools}
\usepackage{cleveref}
\usepackage{tikz}
\usepackage{pgfplots}
\pgfplotsset{width=10cm,compat=1.9}
\usepackage{pgfplotstable}
\biboptions{sort&compress}

\newtheorem{theorem}{Theorem}
\newtheorem{lemma}[theorem]{Lemma}
\newdefinition{remark}{Remark}
\newtheorem{asp}[theorem]{Assumption}
\newproof{pf}{Proof}

\DeclareMathOperator{\Diag}{Diag}
\DeclareMathOperator{\subjectto}{subject~to}

\newcommand{\minimize}{\mathop{\rm minimize}\limits}
\newcommand{\maximize}{\mathop{\rm maximize}\limits}

\journal{Elsevier Journal}

\begin{document}

\begin{frontmatter}



\title{Cardinality-constrained Distributionally Robust Portfolio Optimization}


\author[kk]{Ken Kobayashi\corref{cor1}}
\cortext[cor1]{Corresponding author}
\ead{kobayashi.k.ar@m.titech.ac.jp}
\author[yt]{Yuichi Takano}
\author[kk]{Kazuhide Nakata}

\address[kk]{Department of Industrial Engineering and Economics, School of Engineering, \\ Tokyo Institute of Technology, 2-12-1 Ookayama, Meguro-ku, Tokyo 152-8552, Japan}
\address[yt]{Institute of Systems and Information Engineering, University of Tsukuba, \\ 1-1-1 Tennodai, Tsukuba-shi, Ibaraki 305-8573, Japan}

\begin{abstract} %
This paper studies a distributionally robust portfolio optimization model with a cardinality constraint for limiting the number of invested assets. 
We formulate this model as a mixed-integer semidefinite optimization (MISDO) problem by means of the moment-based ambiguity set of probability distributions of asset returns. 
To exactly solve large-scale problems, we propose a specialized cutting-plane algorithm that is based on bilevel optimization reformulation. 
We prove the finite convergence of the algorithm. 
We also apply a matrix completion technique to lower-level SDO problems to make their problem sizes much smaller. 
Numerical experiments demonstrate that our cutting-plane algorithm is significantly faster than the state-of-the-art MISDO solver SCIP-SDP. 
We also show that our portfolio optimization model can achieve good investment performance compared with the conventional robust optimization model based on the ellipsoidal uncertainty set. 
\end{abstract}

\begin{keyword}
 Portfolio optimization\sep Mixed-integer semidefinite optimization \sep Distributionally robust optimization\sep 
 Cutting-plane algorithm \sep Matrix completion

\end{keyword}

\end{frontmatter}

\section{Introduction}\label{sec:intro}
Portfolio optimization models, originating from the mean-variance portfolio selection pioneered by \citet{Markowitz1952}, have been actively studied by academic researchers and institutional investors. 
In real-world situations of portfolio optimization, we must depend on inaccurate estimates of asset returns, which can lead to lower investment performance. 
Accordingly, robust optimization~\citep{BenTal2002,Bental2009}, which copes with uncertainty about input data, has played an important role in practical portfolio optimization~\citep{fabozzi2010robust,Gregory2011}.
This paper focuses on the distributionally robust optimization approach to portfolio optimization. 

In distributionally robust optimization, optimal solutions are evaluated under the worst-case expectation with respect to a set of probability distributions of uncertain parameters. 
This optimization model was first introduced by \citet{Scarf1958} for an inventory problem. 
Distributionally robust optimization can be viewed as a framework unifying traditional stochastic optimization based on sample average approximation~\citep{Kleywegt2002,Shapiro2003,Shapiro2009} and conventional robust optimization based on uncertainty sets of possible realizations of random variables~\citep{BenTal2002,Bental2009}. 
Theoretical analyses also support the effectiveness of distributionally robust optimization in data-driven decision-making problems under uncertainty~\citep{Gotoh2021,Parys2020}.

Moment-based ambiguity sets, which contain probability distributions of asset returns whose moments satisfy certain conditions, are commonly adopted in distributionally robust portfolio optimization. 
\citet{Ghaoui2003} considered the worst-case value-at-risk over an ambiguity set of probability distributions such that the mean and covariance matrix are elementwise bounded. 
\citet{Ttnc2004} used a similar ambiguity set for solving robust mean-variance portfolio optimization problems. 
\citet{GoIy03} studied a distributionally robust portfolio optimization problem in which asset returns are formed using a linear factor model.
\citet{Popescu2007} assumed that the mean and covariance matrix are exactly known and used an ambiguity set of probability distributions whose first and second moments are equal to them, respectively. 
Similar assumptions were also made in other distributionally robust optimization models~\citep{Natarajan2010,Zymler2011}. 
\citet{Goh2010} studied an ambiguity set of probability distributions whose mean belongs to a convex set and covariance matrix is given. 
As an extension of the elementwise ambiguity sets~\citep{Ghaoui2003,Ttnc2004}, 
  an ellipsoidal ambiguity set for the tuple of the mean and covariance matrix was also used in minimizing worst-case value-at-risk and conditional value-at-risk~\citep{Lotfi2016,Lotfi2018} (see \citet{Rahimian19} for a comprehensive survey on distributionally robust optimization models). 

Unlike these models that require prior knowledge of the region including the mean and covariance matrix, \citet{Delage2010} proposed a data-driven method for defining a moment-based ambiguity set and gave its probabilistic guarantee. 
They also formulated the worst-case loss minimization problem based on this ambiguity set as a semidefinite optimization (SDO) problem, which can be solved in polynomial time with interior-point methods~\citep{Nesterov1994}.  
\citet{Bertsimas2017} proposed a practical bootstrap method to calibrate an ambiguity set of probability distributions from historical data. 

We focus on the distributionally robust portfolio optimization model that is based on the moment-based ambiguity set developed by \citet{Delage2010} with a cardinality constraint for limiting the number of invested assets. 
When the number of invested assets is large, it is difficult for investors to keep track of each asset, and substantial transaction costs are also required~\citep{Mansini2014,Perold1984}. 
Thus, there is a need to control the number of invested assets with the cardinality constraint. 
Due to the moment-based ambiguity set and cardinality constraint, our portfolio optimization model is formulated as a mixed-integer semidefinite optimization (MISDO) problem, which is a computationally challenging optimization problem. 

Conventional methods for solving MISDO problems are mainly classified into two categories: branch-and-bound and cutting-plane algorithms.
Regarding branch-and-bound algorithms, \citet{Gally2017} developed a general-purpose MISDO solver called SCIP-SDP, which combines the branch-and-bound framework of SCIP~\citep{Achterberg2009} with SDO solvers that use interior-point methods. 
Branch-and-bound algorithms have been used for specific applications of MISDO models~\citep{ArFu12,CeAg13,YoKa10}. 
Regarding cutting-plane algorithms, \citet{Coey2020} developed a general-purpose solver called Pajarito for mixed-integer conic optimization problems. 
Cutting-plane algorithms for solving MISDO problems have been applied to eliminating multicollinearity from linear regression models~\citep{Tamura2017} and allocating surgery blocks to operating rooms~\citep{ZhSh18}. 
\citet{Kobayashi2019} devised a branch-and-cut algorithm for solving MISDO problems. 
As shown above, several algorithms have been proposed to solve MISDO problems; however, a specialized algorithm is required to solve large-scale portfolio optimization problems.  

\citet{Bertsimas2021unified} proposed a general framework of cutting-plane algorithms to exactly solve mixed-integer convex optimization problems with logical constraints. 
They reformulated this problem as a bilevel optimization problem composed of lower- and upper-level problems. 
To solve the upper-level problem, they devised a cutting-plane algorithm, which iteratively approximates the objective function by generating cutting planes from the solution to the lower-level problem on the basis of the strong duality theory. 
\citet{Bertsimas2018} applied this algorithm to cardinality-constrained mean-variance portfolio optimization. 
They demonstrated that their algorithm was much faster than state-of-the-art MIO methods when solving large-scale problem instances.
\citet{Kobayashi2021} developed a bilevel cutting-plane algorithm to solve cardinality-constrained mean-CVaR portfolio optimization problems. 
The cutting-plane algorithm~\citep{Bertsimas2021unified} was also applied to sparse principal component analysis~\citep{Bertsimas2020pca} and sparse inverse covariance estimation~\citep{Bertsimas2019covariance}. 

Motivated by these prior studies, we propose a specialized cutting-plane algorithm to exactly solve the cardinality-constrained distributionally robust portfolio optimization problem. 
Each iteration of this cutting-plane algorithm solves an SDO problem, which is the dual of the lower-level problem for generating cutting planes. 
However, the size of this SDO problem depends on the number of all investable assets, so solving it at each iteration is computationally prohibitive when handling a large number of assets. 
To overcome this difficulty, we reduce the size of the lower-level SDO problem through the application of the matrix completion technique~\citep{Fukuda2001,Nakata2003}. 
Notably, the size of the reduced SDO problem depends not on the number of all investable assets but on the cardinality parameter, which is usually set to a small number. 
\textcolor{black}{
As a result, we can efficiently generate cutting planes even when the number of investable assets is very large. 
}
To the best of our knowledge, we are the first to develop an effective algorithm for exactly solving the cardinality-constrained distributionally robust portfolio optimization problem. 
Numerical experiments using real-world datasets 
demonstrate the effectiveness of our method in terms of both computational efficiency and out-of-sample investment performance. 

The main contributions of the present paper are summarized as follows:
\begin{itemize}
    \item We formulate the cardinality-constrained distributionally robust portfolio optimization model with the moment-based ambiguity set as an MISDO problem. 
    \item We develop the cutting-plane algorithm for solving the cardinality-constrained distributionally robust portfolio optimization problem. 
    We also prove that our algorithm outputs a solution with guaranteed global optimality in a finite number of iterations.
    \item \textcolor{black}{By reducing the size of lower-level SDO problems with the matrix completion technique~\citep{Fukuda2001,Nakata2003}, we can efficiently generate cutting planes regardless of the number of all investable assets.
    Numerical results indicate that our algorithm with this acceleration technique is faster than the state-of-the-art MISDO solver SCIP-SDP~\citep{Gally2017}.} 
    \item \textcolor{black}{We demonstrate through numerical experiments that the cardinality constraint can improve the out-of-sample investment performance of the distributionally robust portfolio optimization model. Moreover, our portfolio optimization model is superior in terms of the out-of-sample investment performance to the conventional robust portfolio optimization model with an ellipsoidal uncertainty set of asset returns~\citep{Ben_Tal_1999}.}
\end{itemize}

The remainder of this paper is organized as follows. 
In \Cref{sec:formulation}, we give an MISDO formulation of the cardinality-constrained distributionally robust portfolio optimization model.
In \Cref{sec:cpa}, we present our cutting-plane algorithm for solving the problem.
In \Cref{sec:matrix_completion}, we describe the reduction of the  lower-level SDO problem. 
We report the numerical results in \Cref{sec:experiments} and conclude in \Cref{sec:concl}.

\paragraph*{Notation} 
The set of consecutive integers ranging from $1$ to $N$ is denoted as $[N]\coloneqq\{1,2,\ldots,N\}.$ 
The zero and all-ones vectors of appropriate sizes are written as $\bm 0$ and $\bm 1$, respectively. 
The zero and identity matrices of appropriate sizes are written as $\bm O$ and $\bm I$, respectively. 
The $\Diag(\cdot)$ operator maps the vector to the diagonal matrix. 
The set of all $N\times N$ real symmetric matrices is denoted as $\mathcal{S}^N$. 
For $\bm X, \bm Y \in \mathcal{S}^N$, we write $\bm X \succ \bm Y$ and $\bm X \succeq \bm Y$ if the matrix $\bm X- \bm Y$ is positive definite and positive semidefinite, respectively. 
The standard inner product of matrices $\bm A\coloneqq(A_{nm})\in \mathcal{S}^N$ and $\bm B\coloneqq(B_{nm})\in \mathcal{S}^N$ is defined as $\bm A\bullet \bm B\coloneqq\sum_{n\in [N]}\sum_{m\in [N]}A_{nm}B_{nm}.$ 
The Hadamard product of vectors $\bm a\coloneqq (a_n)\in \mathbb{R}^N$ and $\bm b \coloneqq(b_n)\in \mathbb{R}^N$ is defined as $\bm a \circ \bm b\coloneqq (a_n b_n)\in \mathbb{R}^N$. 

\section{Problem formulation}\label{sec:formulation}
In this section, we formulate the cardinality-constrained distributionally robust portfolio optimization model that we consider in this paper as an MISDO problem. 

\subsection{Cardinality constraint}
Let $\bm x \coloneqq (x_1,x_2,\dots,x_N)^\top$ be a portfolio, where $x_n$ is the investment weight of the $n$th asset. 
Throughout this paper, we consider the set of feasible portfolios:
\begin{equation*}
    \mathcal X \coloneqq \left\{\bm x \in \mathbb R^N\,\middle|\,\sum_{n\in [N]}x_n=1, \quad \bm x \geq 0 \right\}.
\end{equation*}
The nonnegativity constraint on $\bm x$ prohibits short selling. 

Let $k \in [N]$ be a user-defined parameter for limiting the cardinality (i.e., the number of assets to be held). 
We then impose the following cardinality constraint~\citep{Bertsimas1999,Bienstock1996,Perold1984} on portfolio $\bm x$:
\begin{equation}\label{eq:cardinality}
  \|\bm x\|_0 \le k,   
\end{equation}
where $\|\cdot\|_0$ is the $\ell_0$-pseudo-norm (i.e., the number of nonzero entries).
In practice, this constraint is required by investors to reduce their portfolio monitoring and transaction costs. 
\textcolor{black}{The cardinality parameter $k$ can be set according to the acceptable amount of such costs or tuned through the cross-validation of investment performance.}

Let $\bm z \coloneqq (z_1,z_2,\ldots, z_N)^\top$ be a vector composed of binary decision variables for selecting assets, that is, $z_n = 1$ if the $n$th asset is selected, and $z_n=0$ otherwise. 
We also introduce the feasible set corresponding to the cardinality constraint~\eqref{eq:cardinality}: 
\begin{equation*}
\mathcal{Z}_N^k \coloneqq \left\{\bm z\in \{0,1\}^N  ~\middle|~ \sum_{n\in [N]} z_n = k\right\}. 
\end{equation*}
The cardinality constraint \eqref{eq:cardinality} is then represented by the logical implication:
\begin{equation*}
\begin{cases}
    {}
    z_n  = 0 ~\Rightarrow~ x_n = 0 &(\forall n\in [N]), \\
    \bm z \in \mathcal{Z}_N^k.
\end{cases}    
\end{equation*}

\subsection{Moment-based ambiguity set}
We consider a measurable space $(\Omega, \mathcal{F})$. 
Let $\tilde{\bm \xi}: \Omega\to \mathbb{R}^N$ be an $\mathcal{F}$-measurable function ($N$-dimensional random vector) representing the rate of random return of each asset. 
Suppose that $\bm \xi_m \in \mathbb{R}^N$ is an observed historical return for each period $m\in [M]$. 
The sample mean vector and sample covariance matrix are then calculated as
\begin{align}\label{eq:sample_mean_cov}
    \hat{\bm \mu} \coloneqq \frac{1}{M}\sum_{m\in [M]}\bm \xi_m, \qquad
    \hat{\bm \Sigma} \coloneqq \frac{1}{M}\sum_{m\in [M]}(\bm \xi_m-\hat{\bm \mu})(\bm \xi_m-\hat{\bm \mu})^\top.
\end{align}

Let $\mathcal{M}$ be the set of all probability measures in the measurable space $(\Omega, \mathcal{F})$, and $\mathbb{E}_F[\,\cdot\,]$ be the expectation under the probability measure $F \in \mathcal{M}$. 
\citet{Delage2010} considered the moment-based ambiguity set of probability distributions of $\tilde{\bm \xi}$. 
When the support of probability distributions is $\mathbb{R}^N$, this ambiguity set is given by
\begin{equation}
	\mathcal{D}(\hat{\bm \mu} , \hat{\bm \Sigma} , \kappa_1, \kappa_2) \coloneqq \left\{F\in \mathcal{M} \left|
		\begin{array}{l}
		\left(\mathbb{E}_F[\tilde{\bm \xi}]-\hat{\bm \mu} \right)^\top \hat{\bm \Sigma} ^{-1} \left(\mathbb{E}_F[\tilde{\bm \xi}]-\hat{\bm \mu} \right) \leq \kappa_1\\
		\mathbb{E}_F[(\tilde{\bm \xi}-\hat{\bm \mu} )(\tilde{\bm \xi}-\hat{\bm \mu} )^\top] \preceq \kappa_2 \hat{\bm \Sigma} 
		\end{array}
	\right. \right\}, \label{eq:uncertain_set}
\end{equation}
where $\kappa_1\geq 0$ and $\kappa_2\geq 1$ are user-defined ambiguity parameters about $\hat{\bm \mu}$ and $\hat{\bm \Sigma}$, respectively.  
The first condition in Eq.~\eqref{eq:uncertain_set} ensures that the expectation of $\tilde{\bm \xi}$ lies in an ellipsoid of size $\kappa_1$ centered at $\hat{\bm \mu}$.
The second condition in Eq.~\eqref{eq:uncertain_set} indicates that the second central-moment matrix of $\tilde{\bm \xi}$ is bounded above by $\kappa_2  \hat{\bm \Sigma} $ in the sense of the matrix inequality. 

We make the following mild assumption about the moment-based ambiguity set~\eqref{eq:uncertain_set}. 
\begin{asp}\label{asp:1}
$\kappa_1 >0$ and $\hat{\bm \Sigma} \succ \bm O$. 
\end{asp}

\color{black}
\begin{remark}\label{rmk:degenerate}
The sample covariance matrix $\hat{\bm \Sigma}$ can be singular when $M$ is less than $N$ or some asset returns are highly correlated~\citep{Bodnar2016,Gulliksson2019,Pappas2010}. 
In this case, we can use shrinkage estimation methods~\citep{Chen2010,Ledoit2004} to obtain a covariance matrix that satisfies \Cref{asp:1}.
\end{remark}
\color{black}

\subsection{Piecewise-linear utility and loss functions}
Let $\bm \xi \in \mathbb{R}^N$ be a realization of the random return $\tilde{\bm \xi}$.
As in \citet{Delage2010}, we use the following piecewise-linear concave utility function in the portfolio net return $y = \bm \xi^\top \bm x$: 
\begin{equation} \label{eq:utility_func}
	u(y) \coloneqq \min_{\ell \in [L]}\left\{ a^{(\ell)} y + b^{(\ell)}\right\},
\end{equation}
where $a^{(\ell)}, b^{(\ell)} \in \mathbb{R}$ are respectively the slope and intercept of the $\ell$th linear function for $\ell \in [L]$. 
The loss function is then defined as the negative of the utility function~\eqref{eq:utility_func}:
\begin{align*}
    \mathcal{L}(\bm x, \bm \xi) &\coloneqq -\min_{\ell \in [L]} \left\{a^{(\ell)} \bm \xi^\top \bm x +b^{(\ell)}\right\} = \max_{\ell \in [L]} \left\{-a^{(\ell)} \bm \xi^\top \bm x -b^{(\ell)}\right\}.
\end{align*}

\subsection{Portfolio optimization model}
Our objective is to minimize the following worst-case expected loss with respect to an underlying probability distribution $F\in \mathcal D( \hat{\bm{\mu}}, \hat{\bm \Sigma},\kappa_1,\kappa_2)$ of $\tilde{\bm \xi}$: 
\begin{equation}\label{eq:worst_case_loss}
    \max_{F\in \mathcal D( \hat{\bm{\mu}}, \hat{\bm \Sigma},\kappa_1,\kappa_2)} \left\{ \mathbb E_F\left[\mathcal{L}(\bm x,\tilde{\bm \xi})\right] \right\}.
\end{equation}
The $\ell_2$-regularization term $\bm x^\top \bm x$ is also incorporated into the objective from the perspective of robust optimization~\citep{Bertsimas2018,DeMiguel2009,Gotoh2013,Gotoh2011}.

We formulate the cardinality-constrained distributionally robust portfolio optimization model as follows:
\begin{subequations}\label{prob:card_constrained_dist_robust}
\begin{align}
		\underset{\bm x, \bm z}{\minimize} 
		&\quad \frac{1}{2\gamma}\bm x^\top \bm x + 
		\max_{F\in \mathcal D( \hat{\bm{\mu}}, \hat{\bm \Sigma},\kappa_1,\kappa_2)} \left\{ \mathbb E_F\left[\mathcal{L}(\bm x,\tilde{\bm \xi})\right] \right\}  \label{obj}\\
        \subjectto&\quad z_n  = 0 ~\Rightarrow~ x_n = 0 \quad(\forall n\in [N]),\\
        &\quad \bm x\in \mathcal{X},\quad \bm z \in \mathcal{Z}_N^k,    
\end{align}
\end{subequations}
where $\gamma>0$ is a user-defined regularization parameter.
By deriving the dual of the inner maximization problem~\eqref{eq:worst_case_loss} as in \citet{Delage2010}, Problem \eqref{prob:card_constrained_dist_robust} can equivalently be reformulated as the MISDO problem:
\begin{subequations} \label{prob:reg_dist_robust_primal}
\begin{align}
\minimize_{\bm x, \bm P, \bm Q, \bm p, \bm q, r, s, \bm z} &\quad \frac{1}{2\gamma}\bm x^\top \bm x + (\kappa_2 \hat{\bm \Sigma}  - \hat{\bm \mu} \hat{\bm \mu}^\top )\bullet \bm Q + r + \hat{\bm \Sigma} \bullet \bm P - 2 \hat{\bm \mu}^\top \bm p + \kappa_1 s\\
\subjectto &\quad \bm{p} = -\bm q /2 - \bm Q \hat{\bm{\mu}},\\
        &\quad 
        \begin{pmatrix}
        \bm Q & \bm q/2+a^{(\ell)}\bm x/2\\
        (\bm q/2+a^{(\ell)}\bm x/2)^\top &r+ b^{(\ell)}
        \end{pmatrix}\succeq \bm O \quad(\forall \ell \in [L]),\\
            &\quad \begin{pmatrix}
            \bm P &\bm p\\
            \bm p^\top &s
            \end{pmatrix}\succeq \bm O,\\
        &\quad z_n  = 0 ~\Rightarrow~ x_n = 0 
        \quad(\forall n\in [N]),\label{const:logic}\\
        &\quad \bm x\in \mathcal{X},\quad \bm z \in \mathcal{Z}_N^k,
\end{align}
\end{subequations}
where $\bm P, \bm Q \in \mathcal S^N$, $\bm p, \bm q \in \mathbb R^N$, and $r, s \in \mathbb R$ are dual decision variables of the inner maximization problem~\eqref{eq:worst_case_loss}. 

\section{Cutting-plane algorithm}\label{sec:cpa}
In this section, we present our cutting-plane algorithm for solving the cardinality-constrained distributionally robust portfolio optimization problem \eqref{prob:reg_dist_robust_primal}.

\subsection{Bilevel optimization reformulation} 
We extend the method of bilevel optimization reformulation~\citep{Bertsimas2018} to our portfolio optimization problem \eqref{prob:reg_dist_robust_primal}.
Let us denote by $\bm Z\coloneqq \Diag(\bm z)$ a diagonal matrix whose diagonal entries are given by $\bm z$. 
The logical implication \eqref{const:logic} then amounts to replacing $\bm x$ with $\bm Z \bm x$ in Problem \eqref{prob:reg_dist_robust_primal}. 
We now can reformulate Problem \eqref{prob:reg_dist_robust_primal} as a bilevel optimization problem. 
Specifically, the \emph{upper-level problem} is written as the integer optimization problem: 
\begin{equation}\label{prob:master}
    \minimize_{\bm z}\quad f(\bm z)\quad \subjectto \quad \bm z\in \mathcal{Z}_N^k,\end{equation}
and the \emph{lower-level problem} for calculating the objective function is expressed as the semidefinite optimization problem:
\begin{subequations} \label{prob:lower_primal}
\begin{align}
f(\bm z) = \minimize_{\bm x, \bm P, \bm Q, \bm p, \bm q, r, s} &\quad \frac{1}{2\gamma} \bm x^\top \bm x + (\kappa_2 \hat{\bm \Sigma}  - \hat{\bm \mu} \hat{\bm \mu}^\top )\bullet \bm Q + r + \hat{\bm \Sigma} \bullet \bm P - 2 \hat{\bm \mu}^\top \bm p + \kappa_1 s\label{eq:lower_primal_obj}\\
\subjectto &\quad \bm{p} = -\bm q /2 - \bm Q \hat{\bm{\mu}},\\
        &\quad 
        \begin{pmatrix}
        \bm Q & \bm q/2+a^{(\ell)}\bm Z\bm x/2\\
        (\bm q/2+a^{(\ell)}\bm Z\bm x/2)^\top &r+ b^{(\ell)}
        \end{pmatrix}\succeq \bm O \quad(\forall \ell \in [L]), \label{con2:lower_primal}\\
            &\quad \begin{pmatrix}
            \bm P &\bm p\\
            \bm p^\top &s
            \end{pmatrix}\succeq \bm O, \label{con3:lower_primal}\\
        &\quad \bm Z\bm x\in \mathcal{X}.    
\end{align}
\end{subequations}
Note that $(\bm Z\bm x)^\top \bm Z\bm x$ has been replaced by $\bm x^\top \bm x$ in Eq.~\eqref{eq:lower_primal_obj} because $\bm x=\bm Z\bm x$ holds through minimizing $\bm x^\top \bm x$. 

The following theorem yields the dual formulation of Problem \eqref{prob:lower_primal}. 
\begin{theorem}\label{thm:f_dual}
For all $\bm z\in \mathcal{Z}_N^k$, the strong duality holds for Problem \eqref{prob:lower_primal}, and the dual formulation of Problem \eqref{prob:lower_primal} is represented as follows:
\begin{subequations}\label{prob:lower_dual}
\begin{align}
f(\bm z) = \maximize_{\bm \omega, \bm B, \bm \beta, \bm \eta,\bm \lambda,\pi} &\quad -\frac{\gamma}{2} \bm z^\top (\bm \omega\circ \bm \omega) -\sum_{\ell\in [L]} \eta^{(\ell)} b^{(\ell)} + \pi \label{eq:lower_dual_obj}\\
\subjectto & \quad \bm \omega \geq \sum_{\ell\in [L]}a^{(\ell)} \bm \beta^{(\ell)} +\pi\bm 1,\label{dual_const_ineq:1}\\
&\quad \sum_{\ell\in [L]}\bm B^{(\ell)}=\kappa_2 \hat{\bm \Sigma}  -\hat{\bm \mu}\hat{\bm \mu}^\top+ \hat{\bm \mu}\left(\sum_{\ell\in [L]}\bm \beta^{(\ell)}\right)^\top+\left(\sum_{\ell\in [L]}\bm \beta^{(\ell)}\right)\hat{\bm \mu}^\top,\label{dual_const_eq:1}\\
&\quad \sum_{\ell\in [L]}\bm \beta^{(\ell)} = \bm \lambda+\hat{\bm \mu},\label{dual_const_eq:2}\\
&\quad  1 - \sum_{\ell\in [L]}\eta^{(\ell)} = 0,\label{dual_const_eq:3}\\
&\quad  \begin{pmatrix}
     \bm B^{(\ell)} &\bm \beta^{(\ell)}\\
     (\bm \beta^{(\ell)})^{\top} & \eta^{(\ell)}
    \end{pmatrix}\succeq \bm O\quad (\forall \ell \in [L]),\label{dual_const_sd:1}\\
&\quad \begin{pmatrix}
     \hat{\bm \Sigma} &\bm \lambda\\
     \bm \lambda^\top & \kappa_1
    \end{pmatrix}\succeq \bm O\label{dual_const_sd:2},
\end{align}
\end{subequations}
where $\bm \omega\in \mathbb{R}^N$, $\bm B \coloneqq (\bm B^{(\ell)}) \in \mathbb{R}^{N \times N \times L}$, 
$\bm \beta \coloneqq (\bm \beta^{(\ell)}) \in \mathbb{R}^{N \times L}$, 
$\bm \eta\coloneqq(\eta^{(\ell)}) \in \mathbb{R}^L$, $\bm \lambda \in \mathbb{R}^N$, and $\pi\in \mathbb{R}$ are dual decision variables. 
\end{theorem}

\begin{proof}
See \ref{sec:proof_f_dual}.
\end{proof}
In accordance with \Cref{thm:f_dual}, we can extend the definition of $f(\bm z)$ to the optimal objective value of Problem \eqref{prob:lower_dual} for real-valued $\bm z \in [0,1]^N.$ 
As in~\citet{Bertsimas2018}, the function $f(\bm z)$ is convex in $\bm z\in [0,1]^N$, and its subgradient for $\bm z\in \{0,1\}^N$ is given by 
    \begin{equation}\label{eq:subgradient}
        \bm g (\bm z) \coloneqq -\frac{\gamma}{2}\bm \omega^\star(\bm z)\circ\bm \omega^\star(\bm z) \in \partial f(\bm z), 
    \end{equation}
where $\bm \omega^\star(\bm z)$ is an optimal solution of $\bm{\omega}$ to Problem~\eqref{prob:lower_dual}. 
Therefore, Problem~\eqref{prob:lower_dual} with $\hat{\bm z} \in \{0,1\}^N$ provides the following linear underestimator of $f(\bm z)$ for $\bm{z} \in [0,1]^N$: 
\begin{equation}\label{eq:cut}
    f(\bm z) \geq f(\hat{\bm z}) +  \bm g(\hat{\bm z})^\top (\bm z-\hat{\bm z}).
\end{equation}

\subsection{Algorithm description}
We extend the cutting-plane algorithm~\citep{Bertsimas2018} with the aim of solving the upper-level problem~\eqref{prob:master}. 
Let $\theta_{\text{LB}}$ be a lower bound of the optimal objective value of Problem~\eqref{prob:master}; this bound can easily be calculated by solving a continuous relaxation version of Problem~\eqref{prob:reg_dist_robust_primal}. 
Our cutting-plane algorithm starts with the initial feasible region: 
\begin{equation}\label{eq:initial_relax_region}
    \mathcal{F}_1 \coloneqq \{(\bm z,\theta)\in  \mathcal{Z}_N^k\times \mathbb{R}\mid \theta \geq \theta_{\text{LB}}\},
\end{equation}
where $\theta$ is an auxiliary decision variable that corresponds to a lower estimate of $f(\bm z)$. 

At the $t$th iteration~($t\geq 1$), our algorithm solves the surrogate upper-level problem:
\begin{equation}\label{prob:master_relax}
    \minimize_{\bm z, \theta}~\theta\quad \subjectto~(\bm z,\theta) \in \mathcal{F}_t,
\end{equation}
where $\mathcal{F}_t$ is a feasible region at the $t$th iteration such that $\mathcal{F}_t \subseteq \mathcal{F}_1$.
Eq.~\eqref{eq:initial_relax_region} ensures that the objective value of Problem \eqref{prob:master_relax} is bounded below; thus, there exists an optimal solution $(\bm z_t,\theta_t)$ to Problem \eqref{prob:master_relax}. 

We next solve the dual lower-level problem \eqref{prob:lower_dual} with $\bm z = \bm z_t$. 
We thus obtain the function value $f(\bm z_t)$ and its subgradient $\bm g(\bm z_t)$ through Eq.~\eqref{eq:subgradient}. 
If $f(\bm z_t) \le \theta_t + \varepsilon$ holds with sufficiently small $\varepsilon \ge 0$, then $\bm z_t$ is an $\varepsilon$-optimal solution to Problem~\eqref{prob:master}, which means that
\begin{equation*}
    f^\star \le f(\bm z_t) \le f^\star + \varepsilon,
\end{equation*}
where $f^\star$ is the optimal objective value of Problem~\eqref{prob:master}.
In this case, we terminate the algorithm with the $\varepsilon$-optimal solution $\bm z_t$.
Otherwise, to obtain a closer estimate of $f(\bm z)$, we add the constraint~\eqref{eq:cut} with $\hat{\bm z} = \bm z_t$ to the feasible region: 
\begin{equation}\label{eq:cut_subgrad}
    \mathcal{F}_{t+1} \leftarrow \mathcal{F}_t \cap \{(\bm z,\theta)\in \mathcal{Z}_N^k\times \mathbb{R}\mid \theta \geq f(\bm z_t) + \bm g(\bm z_t)^\top (\bm z - \bm z_t)\}.
\end{equation}
Note that this update cuts off the solution $(\bm z_t,\theta_t)$ because $\theta_t < f(\bm z_t)$.

We then set $t \leftarrow t + 1$ and solve the surrogate upper-level problem~\eqref{prob:master_relax} again with the updated feasible region~\eqref{eq:cut_subgrad}. 
We repeat this procedure until we find an $\varepsilon$-optimal solution~$\hat{\bm z}$.
After termination of the algorithm, we can compute the corresponding portfolio by solving the lower-level problem~\eqref{prob:lower_primal} with $\bm z = \hat{\bm z}$. 

Our cutting-plane algorithm is summarized in Algorithm \ref{alg:upper_level_cpa}.
Following \citet{Kobayashi2021}, we can prove the finite convergence of the algorithm.
\begin{theorem}
\Cref{alg:upper_level_cpa} terminates in a finite number of iterations and outputs an $\varepsilon$-optimal solution to Problem~\eqref{prob:master}.
\end{theorem}
\begin{proof}
Let $\{(\bm z_t,\theta_t) \mid t \in [T]\}$ be a sequence of solutions generated by \Cref{alg:upper_level_cpa}. 
Suppose that there exists $t < T$ such that $\bm z_{t} = \bm z_T$. 
Since $(\bm z_T, \theta_T) \in \mathcal{F}_{t+1}$, it follows from Eq.~\eqref{eq:cut_subgrad} that
\[
\text{LB}_{T} = \theta_T \ge f(\bm z_{t}) + \bm g(\bm z_{t})^\top (\bm z_T - \bm z_{t}) = f(\bm z_T) \ge \text{UB}_T, 
\]
which verifies that $\bm z_T$ is an optimal solution to Problem~\eqref{prob:master}. 
Since there are at most a finite number of possible solutions $\bm z \in \mathcal{Z}_N^k$, the algorithm terminates with an $\varepsilon$-optimal solution after a finite number of iterations. 
\end{proof}

 \begin{algorithm}[ht]
 \caption{Cutting-plane algorithm for solving Problem~\eqref{prob:master}} 
 \label{alg:upper_level_cpa}
 \begin{algorithmic} 
 \normalsize
 \STATE \begin{description}
 \item[\textbf{Step 0~\textsf{(Initialization)}}] Let $\varepsilon \geq 0$ be a tolerance for optimality. 
 Define the feasible region~$\mathcal{F}_1$ as in Eq.~\eqref{eq:initial_relax_region}.
 Set $t\leftarrow 1$ and $\text{UB}_0 \leftarrow \infty$.
 \item[\textbf{Step 1 \textsf{(Surrogate Upper-level Problem)}}] Solve Problem \eqref{prob:master_relax}. Let $(\bm z_t, \theta_t)$ be an optimal solution, and set $\text{LB}_t \leftarrow \theta_t$. 
\item[\textbf{Step 2 \textsf{(Dual Lower-level Problem)}}] Solve Problem~\eqref{prob:lower_dual} with $\bm z = \bm z_t$ to calculate $f(\bm z_t)$ and $\bm \omega^\star(\bm z_t)$. If $f(\bm z_t) < \text{UB}_{t-1}$, set $\text{UB}_{t} \leftarrow f(\bm z_t)$ and 
$\hat{\bm z} \leftarrow \bm z_t$; otherwise, set $\text{UB}_{t} \leftarrow \text{UB}_{t-1}$. 
\item[\textbf{Step 3 \textsf{(Termination Condition)}}] If $\text{UB}_t-\text{LB}_t\leq \varepsilon$, terminate the algorithm with the $\varepsilon$-optimal solution~$\hat{\bm z}$. 
\item[\textbf{Step 4 \textsf{(Cut Generation)}}] Calculate $\bm g(\bm z_t)$ as in Eq.~\eqref{eq:subgradient} and update the feasible region as in Eq.~\eqref{eq:cut_subgrad}. 
Set $t\leftarrow t+1$ and return to Step 1.
\end{description}
 \end{algorithmic} 
 \end{algorithm}

\section{Reduction of the lower-level SDO problem}\label{sec:matrix_completion}
Recall that Step 2 of \Cref{alg:upper_level_cpa} solves Problem \eqref{prob:lower_dual}, which is an SDO problem including positive semidefinite constraints~\eqref{dual_const_sd:1}~and~\eqref{dual_const_sd:2} on $(N+1)\times (N+1)$ symmetric matrices. 
It is clearly difficult to directly solve Problem \eqref{prob:lower_dual} when $N$ is very large. 
To remedy this situation, we reduce its problem size by applying the technique of positive semidefinite matrix completion~\citep{Fukuda2001,Nakata2003} to the lower-level SDO problem \eqref{prob:lower_dual}. 

\subsection{Reduced problem formulation}
For any vector $\bm v \coloneqq (v_n) \in \mathbb{R}^N$ and matrix $\bm M \coloneqq (M_{nn'}) \in \mathbb{R}^{N\times N}$, we write the subvector and submatrix corresponding to $\bm z, \bm z' \in \{0,1\}^N$ as 
\begin{align*}
\bm v_{\bm z} & \coloneqq (v_n)_{n \in \mathcal{N}(\bm z)} = (v_n \mid z_n =1) \in \mathbb R^{|\mathcal{N}(\bm z)|}, \\ 
\bm M_{\bm z, \bm z'} & \coloneqq (M_{nn'})_{(n,n') \in \mathcal{N}(\bm z) \times \mathcal{N}(\bm z')} = (M_{nn'} \mid z_n = z_{n'} = 1) \in \mathbb R^{|\mathcal{N}(\bm z)| \times |\mathcal{N}(\bm z')|},  
\end{align*}
where $\mathcal{N}(\bm z) \coloneqq \{n \in [N] \mid z_n = 1\}$.

When $\bm z\in \mathcal{Z}_N^k$, we have $|\mathcal{N}(\bm z)| = k$. 
Then $\bm \omega \in \mathbb R^N$ can be replaced with its subvector $\bm \omega_{\bm z} \in \mathbb R^k$ in the objective~\eqref{eq:lower_dual_obj} as follows:
\begin{equation}\label{eq:l2reg}
\frac{\gamma}{2} \bm z^\top (\bm \omega\circ \bm \omega) = \frac{\gamma}{2} \bm \omega_{\bm z}^\top \bm \omega_{\bm z}.
\end{equation}
The cardinality parameter $k$ is usually much smaller than $N$.
We exploit this problem structure to reduce the size of Problem \eqref{prob:lower_dual}.

A reduced version of Problem \eqref{prob:lower_dual} for $\bm z \in \mathcal{Z}_N^k$ is formulated as 
\begin{subequations}\label{prob:lower_dual_reduced}
\begin{align}
 f'(\bm z) = \maximize_{\bm \omega_{\bm z},\bm B_{\bm z, \bm z} ,\bm \beta_{\bm z}, \bm \eta, \bm \lambda_{\bm z},\pi} &\quad -\frac{\gamma}{2} \bm \omega_{\bm z}^\top \bm \omega_{\bm z} - \sum_{\ell\in [L]}\eta^{(\ell)}b^{(\ell)} + \pi\\
\subjectto & \quad \bm \omega_{\bm z} \geq \sum_{\ell\in [L]}a^{(\ell)} \bm \beta_{\bm z}^{(\ell)} +\pi\bm 1,\\
&\quad \sum_{\ell\in [L]}\bm B_{\bm z,\bm z}^{(\ell)}=\kappa_2 \hat{\bm \Sigma}_{\bm z, \bm z}  -\hat{\bm \mu}_{\bm z}\hat{\bm \mu}_{\bm z}^\top+ \hat{\bm \mu}_{\bm z}\left(\sum_{\ell\in [L]}\bm \beta^{(\ell)}_{\bm z}\right)^\top+\left(\sum_{\ell\in [L]}\bm \beta^{(\ell)}_{\bm z}\right)\hat{\bm \mu}_{\bm z}^\top,\label{lower_dual_reduced_const_eq:1}\\
&\quad \sum_{\ell\in [L]}\bm \beta^{(\ell)}_{\bm z} = \bm \lambda_{\bm z}+\hat{\bm \mu}_{\bm z},\label{lower_dual_reduced_const_eq:2}\\
&\quad  1 - \sum_{\ell\in [L]}\eta^{(\ell)} = 0,\label{lower_dual_reduced_const_eq:3}\\
&\quad  \begin{pmatrix}
     \bm B^{(\ell)}_{\bm z,\bm z} &\bm \beta^{(\ell)}_{\bm z}\\
     (\bm \beta_{\bm z}^{(\ell)})^{\top} & \eta^{(\ell)}
    \end{pmatrix}\succeq \bm O\quad (\ell \in [L]),\label{lower_dual_reduced_const_sd:1}\\
&\quad \begin{pmatrix}
     \hat{\bm \Sigma}_{\bm z,\bm z} &\bm \lambda_{\bm z}\\
     \bm \lambda_{\bm z}^\top & \kappa_1
    \end{pmatrix}\succeq \bm O,\label{lower_dual_reduced_const_sd:2}
\end{align}
\end{subequations}
where $\bm \omega_{\bm z}\in \mathbb{R}^k$, $\bm B_{\bm z,\bm z} \coloneqq (\bm B_{\bm z,\bm z}^{(\ell)}) \in \mathbb{R}^{k \times k \times L}$, 
$\bm \beta_{\bm z} \coloneqq (\bm \beta_{\bm z}^{(\ell)}) \in \mathbb{R}^{k \times L}$, and 
$\bm \lambda_{\bm z}\in \mathbb{R}^k$ are reduced versions of decision variables. 

In the next subsection, we verify that the reduced problem~\eqref{prob:lower_dual_reduced} is equivalent to the original problem~\eqref{prob:lower_dual} in the sense that an optimal solution to the original problem \eqref{prob:lower_dual} can be recovered from an optimal solution to the reduced problem~\eqref{prob:lower_dual_reduced}. 
We also prove that $f'(\bm z)$ defined by the reduced problem~\eqref{prob:lower_dual_reduced} is equal to $f(\bm z)$ defined by the original problem~\eqref{prob:lower_dual}.

\subsection{Equivalence of original and reduced lower-level problems}\label{sec:equiv}

We first focus on the following lemma.
\begin{lemma}\label{lem:bar_beta_ineq} 
Let $(\bm \omega_{\bm z}, \bm B_{\bm z,\bm z}, \bm \beta_{\bm z},\bm \eta, \bm \lambda_{\bm z}, \pi)$ be a feasible solution to Problem \eqref{prob:lower_dual_reduced} for $\bm z \in \mathcal{Z}_N^k$, and 
$\bar{\bm \beta} \coloneqq (\bar{\bm \beta}^{(\ell)}) \in \mathbb{R}^{N \times L}$ be defined as 
\begin{equation}\label{def:beta}
     \left\{
    \begin{aligned}
          &\bar{\bm \beta}^{(\ell)}_{\bm z} \coloneqq \bm \beta_{\bm z}^{(\ell)},\\
    &\bar{\bm \beta}^{(\ell)}_{\bm 1 -\bm z} \coloneqq \hat{\bm \Sigma}_{{\bm 1-\bm z},{\bm z} }(\hat{\bm \Sigma}_{{\bm z},{\bm z} })^{-1}(\bm  \beta_{\bm z} ^{(\ell)}-\eta^{(\ell)}\hat{\bm \mu}_{\bm z} )+\eta^{(\ell)}\hat{\bm \mu}_{\bm 1-\bm z} 
    \end{aligned}
    \right. \quad(\forall \ell \in [L]).
\end{equation}
It then follows that
\begin{align}\label{eq:matrix_ineq}
    \kappa_2 \hat{\bm \Sigma} \succeq \sum_{\substack{\ell\in [L]\\\eta^{(\ell)}> 0}}\frac{1}{\eta^{(\ell)}}(\bar{\bm \beta}^{(\ell)}-\eta^{(\ell)}\hat{\bm \mu})(\bar{\bm \beta}^{(\ell)}-\eta^{(\ell)}\hat{\bm \mu})^\top.
\end{align}
\end{lemma}
\begin{proof}
See \ref{sec:proof_bar_beta_ineq}.
\end{proof}

We next prove that feasible $\bm B$ and $\bm \beta$ for the original problem~\eqref{prob:lower_dual} can be completed from a feasible solution to the reduced problem~\eqref{prob:lower_dual_reduced}. 
\begin{lemma}\label{lem:B_completion}
Let $(\bm \omega_{\bm z}, \bm B_{\bm z,\bm z}, \bm \beta_{\bm z},\bm \eta, \bm \lambda_{\bm z}, \pi)$ be a feasible solution to Problem \eqref{prob:lower_dual_reduced} for $\bm z \in \mathcal{Z}_N^k$, and $\bar{\bm \beta}$ be defined by Eq.~\eqref{def:beta}.  
There then exists 
$\bar{\bm B} \coloneqq (\bar{\bm B}^{(\ell)}) \in \mathbb{R}^{N \times N \times L}$ such that $(\bm B,\bm\beta) = (\bar{\bm B},\bar{\bm\beta})$ satisfies Eqs.~\eqref{dual_const_eq:1} and \eqref{dual_const_sd:1}.
\end{lemma}
\begin{proof}
See \ref{sec:proof_B_completion}.
\end{proof}

We are now in a position to prove our main theorem. 
\begin{theorem}\label{thm:reduction}
Let $(\bm \omega_{\bm z}, \bm B_{\bm z,\bm z}, \bm \beta_{\bm z},\bm \eta, \bm \lambda_{\bm z}, \pi)$ be an optimal solution to Problem \eqref{prob:lower_dual_reduced} for $\bm z \in \mathcal{Z}_N^k$, and $(\bar{\bm \omega}, \bar{\bm \beta}, \bar{\bm \lambda})$ be defined by Eq.~\eqref{def:beta} and 
\begin{align}
    & \left\{
    \begin{aligned}
    &\bar{\bm \omega}_{\bm z} \coloneqq \bm \omega_{\bm z},\\
    &\textstyle \bar{\bm \omega}_{\bm 1 -\bm z} \coloneqq \left[\sum_{\ell\in [L]}a^{(\ell)}\bar{\bm \beta}_{\bm 1 -\bm z}^{(\ell)}+\pi\bm 1\right]_{+},
    \end{aligned}
    \right. \label{eq:def_bar_omega} \\
    & \left\{
    \begin{aligned}
    &\bar{\bm \lambda}_{\bm z} \coloneqq \bm \lambda_{\bm z},\\
    &\bar{\bm \lambda}_{\bm 1 -\bm z} \coloneqq \hat{\bm \Sigma}_{\bm 1 -\bm z, \bm z}(\hat{\bm \Sigma}_{\bm z, \bm z})^{-1}\bm \lambda_{\bm z},
    \end{aligned}
    \right. \label{eq:def_bar_lambda}
\end{align}
where $[\bm v]_{+} \coloneqq (\max \{0, v_n\}) \in \mathbb{R}^N$ for $\bm v \coloneqq (v_n) \in \mathbb{R}^N$.
There then exists $\bar{\bm B}$ such that $(\bar{\bm \omega}, \bar{\bm B}, \bar{\bm \beta}, \bm \eta, \bar{\bm \lambda}, \pi)$ is an optimal solution to Problem \eqref{prob:lower_dual}. 
In addition, we have $f(\bm z) = f'(\bm z)$ for all $\bm z \in \mathcal{Z}_N^k$. 
\end{theorem}
\begin{proof}
See  \ref{sec:proof_reduction}.
\end{proof}

\begin{remark}
Note that Eq.~\eqref{eq:def_bar_omega} defines a minimum-norm solution to Problem \eqref{prob:lower_dual}. 
This solution is known to generate strong cutting planes~\citep{Bertsimas2018}. 
\end{remark}

\begin{remark}
When $\kappa_1=0$, unlike in \Cref{asp:1}, we can set $\mathbb{E}_F[\tilde{\bm \xi}] = \hat{\bm \mu}$ and delete the first condition in Eq.~\eqref{eq:uncertain_set}. 
In this case, we can also prove the theorem corresponding to \Cref{thm:reduction}. 
\end{remark}

From \Cref{thm:reduction}, we can revise Step 2 of \Cref{alg:upper_level_cpa} as follows: 
\begin{description}
\item[\textbf{Step 2 \textsf{(Reduced Dual Lower-level Problem)}}] Solve Problem~\eqref{prob:lower_dual_reduced} with $\bm z = \bm z_t$. 
Let $(\bm \omega_{\bm z}, \bm B_{\bm z,\bm z}, \bm \beta_{\bm z},\bm \eta, \bm \lambda_{\bm z}, \pi)$ be an optimal solution. 
Calculate $f(\bm z_t) = f'(\bm z_t)$ and $\bm \omega^\star(\bm z_t) = \bar{\bm \omega}$ from Eqs.~\eqref{def:beta}~and~\eqref{eq:def_bar_omega}. 
If $f(\bm z_t) < \text{UB}_{t-1}$, set $\text{UB}_{t} \leftarrow f(\bm z_t)$ and 
$\hat{\bm z} \leftarrow \bm z_t$; otherwise, set $\text{UB}_{t} \leftarrow \text{UB}_{t-1}$. 
\end{description}
\textcolor{black}{
In the revised Step 2 of \Cref{alg:upper_level_cpa}, we calculate a subgradient $\bm g(\bm z)$ by solving Problem \eqref{prob:lower_dual_reduced}, whose size depends on $k$ instead of $N$. 
Since $k$ is much smaller than $N$ for practical purposes, we can efficiently calculate a subgradient $\bm g(\bm z)$ even when $N$ is very large.
}

\begin{remark}\label{rmk:linear_constarint}
Our cutting-plane algorithm can be extended to a feasible set of portfolios with additional linear constraints: 
\begin{equation*}
    \mathcal X' \coloneqq \left\{\bm x \in \mathbb R^N\,\middle|\,\sum_{n\in [N]}x_n=1, \quad \bm x \geq 0, \quad \bm C\bm x \leq \bm d \right\},
\end{equation*}
where $\bm C \in \mathbb{R}^{J\times N}$ and $\bm d\in \mathbb{R}^J$ are given constants (see, e.g., \citet{Bertsimas2018} and \citet{Kobayashi2021} for details). 
\end{remark}
\color{black}
\begin{remark}\label{rmk:buy-in}
As discussed in \citet[Section 3.2]{Bertsimas2018}, our cutting-plane algorithm can cope with the buy-in threshold constraints: 
\begin{equation*}
    x_n = 0 \quad \text{or}\quad v_n \le x_n \quad (\forall n\in [N]),
\end{equation*}
where $\bm v\coloneqq (v_1, v_2,\dots, v_N)^{\top} \in \mathbb{R}^N$ is the given vector representing the minimum investment thresholds. 
Indeed, these constraints can be reduced to a linear constraint (i.e., $\bm Z\bm v \le \bm x$),
which can be imposed on Problem~\eqref{prob:card_constrained_dist_robust}.
\end{remark}
\color{black}

\section{Numerical experiments}\label{sec:experiments}
In this section, we report on the numerical results from evaluating the efficacy of our method for distributionally robust portfolio optimization with the cardinality constraint.
We first examine the computational efficiency of our cutting-plane algorithm and then demonstrate the out-of-sample investment performance of our portfolio optimization model. 
All experiments were conducted on a Ubuntu 22.04 PC with an Intel Xeon Silver 4210R CPU (2.40 GHz) and 64 GB of memory. 

\textcolor{black}{\Cref{tab:datalist} lists the datasets used in our experiments, where $N$ is the number of assets. 
From Yahoo Finance\footnote{\url{https://finance.yahoo.co.jp}}, 
we downloaded a historical dataset (\texttt{nikkei225}) of Japanese stock returns, where the top 30 companies were selected from the Nikkei 225 index according to market capitalization as of December 2020. 
From the data library on the website of \mbox{Kenneth} R.~French\footnote{\url{https://mba.tuck.dartmouth.edu/pages/faculty/ken.french/}}, 
we downloaded two historical datasets (\texttt{ind49} and \texttt{sbm100}) of US stock returns. 
From the Kaggle datasets\footnote{\url{https://www.kaggle.com/camnugent/
sandp500}}, 
we downloaded a historical dataset (i.e., \texttt{s\&p500}) of US stock returns, where the 32 companies including missing values were omitted from the S\&P500 index. 
From the OR-Library~\citep{Beasley1990,Chang2000}, we downloaded two datasets (\texttt{port2} and \texttt{port5}) of the sample estimates $\hat{\bm \mu}$ and $\hat{\bm \Sigma}$. 
}
\begin{table}[ht]
\normalsize
    \centering
    \caption{Dataset description}
    \label{tab:datalist}
    \begin{tabular}{lrl}
        \toprule
        Abbr. &$N$ &Original dataset  \\ \midrule
        \texttt{nikkei225} &30 & 
30 companies in the Japanese stock index\\
\texttt{ind49} &49 &49 Industry Portfolios\\
        \texttt{port2} &85 &port2 (Portfolio optimization: Single period)\\
        \texttt{sbm100} &100 & 100 Portfolios Formed on Size and Book-to-Market\\
        \texttt{port5} &225 &port5 (Portfolio optimization: Single period)\\
        \texttt{s\&p500} &468 & 
468 companies in the S\&P 500 index\\
        \bottomrule
    \end{tabular}
\end{table}

\subsection{Computational efficiency}\label{sec:exp_alg}
We evaluate the computational efficiency of our cutting-plane algorithm by comparing it with other MISDO algorithms.

\subsubsection{Experimental design}\label{sec:expdes1}
\textcolor{black}{
We used the five datasets: \texttt{ind49}, \texttt{port2}, \texttt{sbm100}, \texttt{port5}, and \texttt{s\&p500}, where the \texttt{nikkei225} dataset was omitted because its problem instances were solved by all methods in a very short time.
For the \texttt{ind49} and \texttt{sbm100} datasets, we used monthly data from January 2010 to December 2019 to calculate the sample estimates $\hat{\bm \mu}$ and $\hat{\bm \Sigma}$ of asset returns~\eqref{eq:sample_mean_cov}. 
For the \texttt{port2} and \texttt{port5} datasets, we multiplied the sample estimates $\hat{\bm \mu}$ and $\hat{\bm \Sigma}$ by 100 and 10,000, respectively, to be consistent with the other datasets.
For the \texttt{s\&p500} dataset, we calculated the sample estimates $\hat{\bm \mu}$ and $\hat{\bm \Sigma}$ from daily data from February 8th, 2013 to February 7th, 2018.}

As the utility function~\eqref{eq:utility_func}, we used a piecewise-linear approximation of the normalized exponential utility function~\citep{Ingersoll1987}:
\begin{equation}\label{eq:exp_utility}
    \tilde{u}(y)\coloneqq\frac{\mu_{\max}(1-\exp(-\alpha y/\mu_{\max}))}{\alpha}, 
\end{equation}
where $\mu_{\max}$ is the maximum entry of the sample mean vector $\hat{\bm \mu}$, and $\alpha>0$ is a risk-aversion parameter. 
We set $\alpha=10$ and used three tangent lines (i.e., $L=3$) at $y \in \{0, \mu_{\max}/2, \mu_{\max}\}$ for piecewise-linear approximation as shown in \ref{apd:utility}. 

We compare the computational performance of the following methods for solving the MISDO problem~\eqref{prob:reg_dist_robust_primal}:
\begin{description}
    \item[SCIP:] MISDO solver SCIP-SDP\footnote{\url{http://www.opt.tu-darmstadt.de/scipsdp/}}~\citep{Gally2017},
    \item[CPA:] our cutting-plane algorithm (\Cref{alg:upper_level_cpa}),
    \item[CPA${}_+$:] our cutting-plane algorithm (\Cref{alg:upper_level_cpa}) with the problem reduction~(i.e., Step 2 in \Cref{sec:equiv}).
\end{description}
These methods used Mosek\footnote{\url{https://www.mosek.com/}} 9.2.40 to solve SDO problems.
CPA and CPA${}_+$ were implemented in Python 3.7 with Gurobi Optimizer\footnote{\url{https://www.gurobi.com/}} 8.1.11 to solve the surrogate upper-level problem~\eqref{prob:master_relax}, where the lazy constraint callback was used to add cutting planes during the branch-and-bound procedure.
We set $\varepsilon =10^{-5}$ as the tolerance for optimality. 
The computation of each method was terminated if it did not finish within 3,600~s.
In these cases, the results obtained within 3,600~s were taken as the final outcome.

The following column labels are used in \Cref{tab:result_k}. 
\begin{description}
    \item[Obj:] objective value of the obtained best feasible solution, 
    \item[GAP(\%):] absolute difference between lower and upper bounds on the optimal objective value divided by the upper bound, 
    \item[Time:] computation time in seconds, 
    \item[\#Cuts:] number of cutting planes generated by the cutting-plane algorithms, 
    \item[\#Nodes:] number of nodes explored in the branch-and-bound procedure. 
\end{description}
Note that the best values of ``Time'' are indicated in bold for each problem instance, and those of ``Obj'' are also indicated in bold for the \texttt{port2} and \texttt{s\&p500} datasets.
Also, ``OM'' in the Obj column indicates that the computation did not start due to the out-of-memory condition. 

\subsubsection{Results of computational performance} 
\color{black}
\Cref{tab:result_k} gives the numerical results of each method for the cardinality parameter $k\in \{5,15,25\}$. 
We set $\gamma= 10/\sqrt{N}$ for the $\ell_2$-regularization term and $(\kappa_1,\kappa_2)=(1,4)$ for the ambiguity set. 
Additional results for $k\in \{10, 20\}$ are given in \ref{apd:c}.

We first focus on our cutting-plane algorithms (i.e., CPA and CPA${}_+$). 
Both CPA and CPA${}_+$ failed to finish solving many problem instances for the \texttt{port2} and \texttt{s\&p500} dataset; however, CPA${}_+$ solved other problem instances to optimality much faster than did CPA, especially for large $N$. 
For the \texttt{sbm100} dataset with $k=5$, although CPA was terminated due to the time limit, CPA${}_+$ solved the problem instance completely in 2.7 s. 
These results verify the effectiveness of our matrix-completion-based problem reduction of the dual lower-level SDO problem. 
 
Next, we compare our cutting-plane algorithms with the MISDO solver SCIP. 
CPA${}_+$ was always the fastest when it finished the computations within the time limit. 
SCIP returned incorrect optimal objective values for the \texttt{ind49} dataset with $k=15$ due to numerical instabilities. 
For the \texttt{port2} dataset with $k\in \{5,15\}$, all the three methods failed to complete the computations within the time limit, but CPA${}_+$ found solutions of better quality than SCIP and CPA. 

For the \texttt{port5} dataset, SCIP did not finish solving even a first continuous relaxation problem within the time limit, thus failing to provide a feasible solution for all $k\in\{5,15,25\}$. 
CPA did not start computations due to the out-of-memory condition. 
In contrast, CPA${}_+$ succeeded in computing solutions with guaranteed optimality for all $k\in\{5,15,25\}$. 

For the \texttt{s\&p500} dataset involving 468 investable assets, only CPA${}_+$ worked normally to find feasible solutions.  
The optimality gaps attained with CPA${}_+$ for $k \in \{5,15,25\}$ were 19.9\%, 10.0\%, and 5.5\%, respectively, and this gap will be smaller if longer time can be spent on the computation. 
These results support the potential of CPA${}_+$ to give good-quality solutions to large problem instances with a limited memory capacity. 

The computation time of our cutting-plane algorithms tended to be shorter for larger $k$. 
For the \texttt{port5} dataset, CPA${}_+$ finished the computations in 164.6 s, 21.7 s, and 21.1 s for $k\in \{5,15,25\}$, respectively. 
The numbers of cutting planes and explored nodes were much smaller for $k=25$ than for $k=5$, which is a part of the reason why CPA${}_+$ is faster with larger $k$. 

While the number of investable assets is smaller in the \texttt{port2} dataset than in the \texttt{sbm100} and \texttt{port5} datasets, the computation time of CPA${}_+$ was much longer for the \texttt{port2} dataset. 
We checked log files of Gurobi and found that the optimality gap of CPA${}_+$ was quite large from an early stage for the \texttt{port2} dataset. 
\Cref{tab:result_k} shows that the numbers of cutting planes and explored nodes of CPA${}_+$ were very large for the \texttt{port2} dataset.  
We also noticed that without the cardinality constraint, optimal portfolios were more diversified for the \texttt{port2} dataset than for the \texttt{sbm100} and \texttt{port5} datasets. 
For these reasons, long computation times were required by CPA${}_+$ for the \texttt{port2} dataset. 

\color{black}
\begin{table}[p]
    \centering
        \caption{Numerical results with $\gamma=10/\sqrt{N}$ and $(\kappa_1,\kappa_2)= (1,4)$ for $k\in \{5,15,25\}$}
        \label{tab:result_k}
     \footnotesize
      \renewcommand{\arraystretch}{0.8}
\begin{tabular}{lrrlrrrrrr}
\toprule
Data &$N$& $k$ 	&Method	 &Obj	 &Gap(\%)	&Time 	&\#Cuts 	&\#Nodes\\ \midrule
\texttt{ind49}	&49	&5	&SCIP	&3.108	&0.0	&1365.8	&---	&61\\
	&	&	&CPA	&3.108	&24.1	&$>$3600	&$>$167	&$>$785\\
	&	&	&CPA${}_+$	&3.108	&0.0	&\textbf{40.4}	&211	&1000\\ \cmidrule{3-9} 
	&	&15	&SCIP	&3.034	&0.0	&1049.2	&---	&49\\
	&	&	&CPA	&3.033	&0.0	&127.7	&6	&1\\
	&	&	&CPA${}_+$	&3.033	&0.0	&\textbf{6.6}	&6	&1\\ \cmidrule{3-9} 
	&	&25	&SCIP	&3.033	&0.0	&542.3	&---	&35\\
	&	&	&CPA	&3.033	&0.0	&106.6	&5	&1\\
	&	&	&CPA${}_+$	&3.033	&0.0	&\textbf{16.9}	&5	&1\\ \midrule 
\texttt{port2}	&85	&5	&SCIP	&2.217	&27.5	&$>$3600	&---&$>$14\\
	&	&	&CPA	&2.181	&100.0	&$>$3600	&$>$28	&$>$110\\
	&	&	&CPA${}_+$	&\textbf{2.027}	&71.7	&$>$3600	&$>$16645	&$>$155410\\ \cmidrule{3-9} 
	&	&15	&SCIP	&1.859	&13.5	&$>$3600	&---&$>$14\\
	&	&	&CPA	&2.009	&96.4	&$>$3600	&$>$24	&$>$51\\
	&	&	&CPA${}_+$	&\textbf{1.637}	&19.4	&$>$3600	&$>$3027	&$>$51474\\ \cmidrule{3-9} 
	&	&25	&SCIP	&1.783	&9.9	&$>$3600	&---&$>$18\\
	&	&	&CPA	&\textbf{1.607}	&0.0	&3281.5	&21	&6\\
	&	&	&CPA${}_+$	&\textbf{1.607}	&0.0	&\textbf{77.3}	&22	&12\\ \midrule 
 \texttt{sbm100}	&100	&5	&SCIP	&4.493	&12.4	&$>$3600	&---&$>$5\\
	&	&	&CPA	&3.940	&0.3	&$>$3600	&$>$12	&$>$1\\
	&	&	&CPA${}_+$	&3.940	&0.0	&\textbf{2.7}	&13	&8\\ \cmidrule{3-9} 
	&	&15	&SCIP	&3.972	&0.9	&$>$3600	&---&$>$5\\
	&	&	&CPA	&3.935	&0.0	&1810.1	&6	&1\\
	&	&	&CPA${}_+$	&3.935	&0.0	&\textbf{5.2}	&5	&1\\ \cmidrule{3-9} 
	&	&25	&SCIP	&4.595	&14.3	&$>$3600	&---&$>$5\\
	&	&	&CPA	&3.935	&0.0	&1558.0	&5	&1\\
	&	&	&CPA${}_+$	&3.935	&0.0	&\textbf{16.8}	&5	&1\\  \midrule
\texttt{port5}	&225	&5	&SCIP	&$\infty$	&100.0	&$>$3600	&---&$>$1\\
    &   &   &CPA  & OM &--- &---&---&---\\
	&	&	&CPA${}_+$	&2.812	&0.0	&\textbf{164.6}	&840	&9285\\ \cmidrule{3-9} 
	&	&15	&SCIP	&$\infty$	&100.0	&$>$3600	&---&$>$1\\
	    &   &   &CPA  & OM &--- &---&---&---\\
	&	&	&CPA${}_+$	&2.677	&0.0	&\textbf{21.7}	&19	&34\\ \cmidrule{3-9} 
	&	&25	&SCIP	&$\infty$	&100.0	&$>$3600	&---&$>$1\\
	    &   &   &CPA  & OM &--- &---&---&---\\
	    &	&	&CPA${}_+$	&2.677	&0.0	&\textbf{21.1}	&6	&1\\ \midrule 
 \texttt{s\&p500}	&468	&5	&SCIP	&$\infty$	&---	&$>$3600	&---&$>$1\\
	    &   &   &CPA  & OM &--- &---&---&---\\
	    &	&	&CPA${}_+$	&\textbf{1.058}	&19.9	&$>$3600	&$>$8747	&$>$382504\\ \cmidrule{3-9} 
	&	&15	&SCIP	&$\infty$	&100.0	&$>$3600	&---&$>$1\\
	    &   &   &CPA  & OM &--- &---&---&---\\
	    &	&	&CPA${}_+$	&\textbf{0.833}	&10.0	&$>$3600	&$>$2948	&$>$157613\\ \cmidrule{3-9} 
	&	&25	&SCIP	&$\infty$	&100.0	&$>$3600	&---&$>$1\\
	    &   &   &CPA  & OM &--- &---&---&---\\
	    &	&	&CPA${}_+$	&\textbf{0.789}	&5.5	&$>$3600	&$>$1079	&$>$48047\\ 
\bottomrule
\end{tabular}
\end{table}

\subsection{Out-of-sample investment performance}\label{sec:exp_oos}
We investigate the investment performance of our cardinality-constrained distributionally robust portfolio optimization model in a practical situation. 
\subsubsection{Experimental design}\label{sec:exp_oos_ed}
\color{black}
For the purpose of comparison, we consider the robust portfolio optimization model~\citep{Ben_Tal_1999} with the cardinality constraint: 
\begin{subequations}\label{prob:RO}
\begin{align}
    \maximize_{\bm x, \bm z} \quad &\min_{\bm \xi\in \mathcal{U}(\hat{\bm \mu}, \hat{\bm \Sigma}, \delta)}~\bm \xi^\top \bm x\\
    \subjectto \quad & z_n  = 0 ~\Rightarrow~ x_n = 0 \quad(\forall n\in [N]),\\
        \quad & \bm x\in \mathcal{X},\quad \bm z \in \mathcal{Z}_N^k,
\end{align}
\end{subequations}
where the ellipsoidal uncertainty set of asset returns is defined as   
\begin{equation*}
    \mathcal{U}(\hat{\bm \mu}, \hat{\bm \Sigma}, \delta) \coloneqq \left\{\bm \xi \in \mathbb{R}^N \mid (\bm \xi-\hat{\bm \mu})^\top \hat{\bm \Sigma}^{-1}(\bm \xi-\hat{\bm \mu}) \leq \delta \right\},
\end{equation*}
and $\delta \ge 0$ is a user-defined uncertainty parameter. 
This model can be posed as a mixed-integer second-order cone optimization problem~\citep{Ben_Tal_1999}.

We compare the out-of-sample investment performance of portfolios determined with the following optimization models:
\begin{description}
    \item[EW:] equally weighted portfolio (i.e., $x_n = 1/N$ for $n\in [N]$);
    \item[RO:] cardinality-constrained robust optimization model~\eqref{prob:RO}, which was solved using Gurobi Optimizer 8.1.11; 
    \item[DR:] cardinality-constrained distributionally robust optimization model~\eqref{prob:reg_dist_robust_primal}, which was solved using our cutting-plane algorithm CPA${}_+$. 
\end{description}
For the DR model, we tuned the ambiguity parameters $(\kappa_1, \kappa_2)$ on the basis of the 80\% confidence region estimated by the tailored bootstrap method~\citep{Bertsimas2017} with 100,000 bootstrap samples. 
We set $\gamma=10/\sqrt{N}$ for the $\ell_2$-regularization term and the utility function~\eqref{eq:utility_func} in the same way as in \Cref{sec:expdes1}.
For the RO model, we set $\delta =  \kappa_1$ for the uncertainty parameter.

We evaluated the out-of-sample investment performance on the basis of a rolling horizon strategy. 
For the \texttt{nikkei225}, \texttt{ind49}, and, \texttt{sbm100} datasets, we downloaded weekly data of stock returns from January 2000 to December 2020. 
We solved portfolio optimization models with sample estimates $\hat{\bm \mu}$ and $\hat{\bm \Sigma}$ in the first training period (156 weeks). 
Only for the \texttt{sbm100} dataset, we used the shrinkage estimation method~\citep{Chen2010} of $\hat{\bm \Sigma}$ to make it positive definite. 
We then calculated weekly returns by applying the obtained portfolios to the subsequent testing period (52 weeks). 
We repeated this process at intervals of 52 weeks until the end of the entire data period. 

The following column labels are used in \Cref{tab:oos_ind49,tab:oos_sbm100,tab:oos_nikkei}:
\begin{description}
    \item[Ave:] average portfolio return, 
    \item[Std:] standard deviation of portfolio returns, 
    \item[VaR:] 90\%-value-at-risk~\citep{Jorion2006} of portfolio losses,
    \item[CVaR:] 90\%-conditional value-at-risk~\citep{Rockafellar2002} of portfolio losses,
    \item[SR:] Sharpe ratio of portfolio returns,  
\end{description}
where portfolio returns are expressed as percentages, and portfolio losses are defined by the negative portfolio returns. 
Note that the best values of each performance measure are indicated in bold for each dataset.
When calculating the Sharpe ratio, we set the risk-free rate to $0.123$\% for the \texttt{ind49} and \texttt{sbm100} datasets and to zero for the \texttt{nikkei225} dataset according to the yields on 10-year US and Japanese government bonds\footnote{\url{https://www.investing.com/}} as of January 2000. 

\subsubsection{Results of investment performance}
\Cref{tab:oos_ind49,tab:oos_sbm100,tab:oos_nikkei} give the out-of-sample investment performance for the \texttt{nikkei225}, \texttt{ind49}, and \texttt{sbm100} datasets with the cardinality parameter $k\in \{5,N\}$.  
Results of optimal investment weights are given in \ref{apd:d}.

For the \texttt{nikkei225} dataset ($N=30$), the EW model had the highest average return and Sharpe ratio, whereas the DR models attained the smallest values of risk measures (i.e., standard deviation, 90\%-VaR, and 90\%-CVaR). 
For the \texttt{ind49} dataset ($N=49$), the DR models were clearly superior to the other models in the risk measures, and the DR model with $k=5$ achieved the highest Sharpe ratio. 
Also for the \texttt{sbm100} dataset ($N=100$), the DR model with $k=5$ outperformed the other models in the risk measures and the Sharpe ratio. 

We can see from these results that distributionally robust portfolio optimization models have an advantage of reducing risks, which is consistent with the implication provided by \citet{Gotoh2021}. 
Note also that the cardinality constraint improved average returns of DR models for all the datasets. 
Consequently, especially when the number of investable assets is large, our portfolio optimization model is capable of realizing low-risk high-return investments in terms of the Sharpe ratio. 

\begin{table}[ht]
    \centering
     \renewcommand{\arraystretch}{0.8}
        \caption{Out-of-sample investment performance for the \texttt{nikkei225} dataset ($N=30$)}
        \label{tab:oos_nikkei}
        \small
    \begin{tabular}{lrrrrrrrrr}
    \toprule
Model &$k$	                &Ave(\%)	&Std(\%)  &VaR(\%)	&CVaR(\%) &SR	    \\ \midrule
EW	               &        &$\textbf{0.216}$	&$2.806$		&$6.918$	&$10.561$ &$\textbf{0.077}$\\  \midrule
RO	   &5       &$0.201$	&$3.176$	&$7.964$	&$11.437$ &$0.063$\\ 
      &$N$	    &$0.194$	&$2.965$	&$7.332$	&$10.863$ &$0.065$\\ \midrule
DR                  &5	    &$0.191$	&$2.692$	&$6.295$	&$10.452$ &$0.071$\\ 
                  &$N$    	&$0.179$	&$\textbf{2.544}$	&$\textbf{5.830}$	&$\textbf{10.023}$ &$0.070$\\ \bottomrule 
    \end{tabular}
\end{table}
\begin{table}[ht]
    \centering
        \caption{Out-of-sample investment performance for the \texttt{ind49} dataset ($N=49$)}
        \label{tab:oos_ind49}
        \small
             \renewcommand{\arraystretch}{0.8}
        \begin{tabular}{lrrrrrrrrr}
    \toprule
Model &$k$	                &Ave(\%)	&Std(\%)  &VaR(\%)	&CVaR(\%) &SR	    \\ \midrule
EW	               &        &$0.212$	&$2.701$	&$8.219$	&$11.515$&$0.033$\\   \midrule
RO	   &5       &$\textbf{0.238}$	&$2.873$		&$7.836$	&$11.398$	&$0.040$\\  
      &$N$	    &$0.236$	&$2.771$	&$7.931$	&$10.929$ &$0.041$\\  \midrule
DR                  &5	    &$0.220$	&$2.044$	&$5.361$	&$9.731$&$\textbf{0.047}$\\  
                  &$N$    	&$0.212$	&$\textbf{2.031}$	&$\textbf{5.022}$	&$\textbf{9.643}$&$0.044$\\  \bottomrule 
    \end{tabular}
\end{table}
\begin{table}[ht]
    \centering
        \caption{Out-of-sample investment performance for the \texttt{sbm100} dataset ($N=100$)}
    \label{tab:oos_sbm100}
    \small
    \renewcommand{\arraystretch}{0.8}
    \begin{tabular}{lrrrrrrrrr}
    \toprule
Model &$k$	                &Ave(\%)	&Std(\%)  &VaR(\%)	&CVaR(\%) &SR	    \\ \midrule
EW	               &        &$0.211$	&$2.982$		&$8.873$	&$12.292$ &$0.029$\\  \midrule
RO	   &5       &$\textbf{0.212}$	&$2.966$		&$9.232$	&$11.960$ &$0.030$\\  
      &$N$	                &$0.204$	&$2.864$	&$8.519$	&$11.912$ &$0.028$\\  \midrule
DR                  &5	    &$0.198$	&$\textbf{2.318}$	&$\textbf{6.019}$	&$\textbf{9.876}$ &$\textbf{0.032}$\\  
                  &$N$    	&$0.191$	&$2.340$	&$6.373$	&$10.120$ &$0.029$\\ \bottomrule 
    \end{tabular}
\end{table}

\color{black}
\section{Conclusion}\label{sec:concl}
We considered the moment-based distributionally robust portfolio optimization model with the cardinality constraint. 
Due to the cardinality constraint, we formulated this model as an MISDO problem, which is very hard to solve exactly when the number of investable assets is large. 
We reformulated the problem as a bilevel optimization problem and devised a specialized cutting-plane algorithm for solving the upper-level problem. 
\textcolor{black}{
We also improved computational efficiency of generating cutting planes by reducing the size of the dual lower-level SDO problem through the effective use of the positive semidefinite matrix completion~\citep{Fukuda2001,Nakata2003}.}

The computational results indicate that our cutting-plane algorithm was significantly accelerated by matrix-completion-based problem reduction. 
Our algorithm was also faster than the state-of-the-art MISDO solver SCIP-SDP when it finished computations within the time limit. 
Our algorithm succeeded in giving an optimal solution within a few minutes to problem instances involving 225 assets. 
It also found a feasible solution of good quality to large-sized problem instances involving 468 assets. 
\textcolor{black}{
In addition, the cardinality constraint improved the out-of-sample Sharpe ratio of the distributionally robust portfolio optimization model. 
As a result, our portfolio optimization model outperformed the conventional robust portfolio optimization model~\citep{Ben_Tal_1999} in many performance measures. 
}

For future work, we expect that our cutting-plane algorithm can be extended to other distributionally robust optimization models. 
For example, various robust portfolio optimization models have been proposed~\citep{fabozzi2010robust}, and discrepancy- and kernel-based ambiguity sets have been used~\citep{Bertsimas2020_from_predictive,Zhao2018} as an alternative to the moment-based ambiguity set of probability distributions. 
\textcolor{black}{
It is also required to extend our cutting-plane algorithm such that it can deal with short selling, which has been considered in the literature~\citep{Bodnar2017,Javed2020,Khodamoradi_2020}.}
As another direction, it is promising to integrate the matrix completion technique into general-purpose MISDO algorithms~\citep{Coey2020,Gally2017,Kobayashi2019} for better computational performance.

\bibliographystyle{apalike}
\bibliography{bibliography}

\appendix
\section{Proof of \Cref{thm:f_dual}}\label{sec:proof_f_dual}
The Lagrange function of Problem~\eqref{prob:lower_primal} is expressed as
\begin{align*}
     &\mathcal{L}(\bm x,\bm P, \bm Q, \bm p, \bm q, r,s;\bm \alpha, \bm B,\bm \beta, \bm \eta, \bm \Lambda, \bm \lambda, \nu, \pi, \bm \rho)\\
     &\coloneqq \frac{1}{2\gamma} \bm x^{\top} \bm x + (\kappa_2 \hat{\bm \Sigma}  - \hat{\bm \mu} \hat{\bm \mu}^\top )\bullet \bm Q + r + \hat{\bm \Sigma} \bullet \bm P  - 2 \hat{\bm \mu}^\top \bm p + \kappa_1 s\\
     &\quad - \bm \alpha^\top (\bm{p} +\bm q /2 +\bm Q \hat{\bm{\mu}})\\
     &\quad - \sum_{\ell\in [L]} 
     \begin{pmatrix}
     \bm B^{(\ell)} &\bm \beta^{(\ell)}\\
     (\bm \beta^{(\ell)})^\top & \eta^{(\ell)}
    \end{pmatrix}\bullet \begin{pmatrix}
        \bm Q & \bm q/2+a^{(\ell)}\bm Z\bm x/2\\
        (\bm q/2+a^{(\ell)}\bm Z\bm x/2)^\top &r+ b^{(\ell)}
        \end{pmatrix}\\
        &
        \quad- \begin{pmatrix}
     \bm \Lambda &\bm \lambda\\
     \bm \lambda^\top & \nu
    \end{pmatrix}
    \bullet \begin{pmatrix}
            \bm P &\bm p\\
            \bm p^\top &s
            \end{pmatrix}
        - \pi (\bm 1^\top \bm Z \bm x-1) - \bm \rho^\top \bm Z\bm x,
\end{align*}
where $\bm \alpha \in \mathbb{R}^N$, $\begin{pmatrix}
     \bm B^{(\ell)} &\bm \beta^{(\ell)}\\
     (\bm \beta^{(\ell)})^\top & \eta^{(\ell)}
    \end{pmatrix}\succeq \bm O~(\ell \in [L])$, $\begin{pmatrix}
     \bm \Lambda &\bm \lambda\\
     \bm \lambda^\top & \nu
    \end{pmatrix}\succeq \bm O$, $\pi \in \mathbb{R}$, and $\bm \rho \geq \bm 0$ are Lagrange multipliers. 
The Lagrange dual of Problem~\eqref{prob:lower_primal} is then posed as
 \begin{align}\label{prob:lagrange_dual}
    \max_{\bm \alpha, \bm B, \bm \beta, \bm \eta, \bm \Lambda, \bm \lambda,\nu,\pi, \bm \rho}&~\min_{\bm x, \bm P, \bm Q, \bm p, \bm q, r, s}~\mathcal{L}(\bm x, \bm P, \bm Q, \bm p, \bm q, r,s;\bm \alpha, \bm B,\bm \beta, \bm \eta, \bm \Lambda, \bm \lambda, \nu, \pi,\bm \rho).
\end{align}
Now, let us focus on the inner minimization problem: 
\begin{equation}\label{prob:relaxed_prob}
    \min_{{\bm x, \bm P, \bm Q, \bm p, \bm q, r, s}}~\mathcal{L}(\bm x, \bm P, \bm Q, \bm p, \bm q, r,s;\bm \alpha, \bm B,\bm \beta, \bm \eta, \bm \Lambda, \bm \lambda, \nu, \pi, \bm \rho).
\end{equation}
Note that Problem \eqref{prob:relaxed_prob} is an unconstrained convex quadratic optimization problem, and its objective function is linear in $(\bm P, \bm Q, \bm p, \bm q, r,s)$.
Since Problem \eqref{prob:relaxed_prob} must be bounded, the Lagrange multipliers are required to satisfy the following conditions:
\begin{align}
    \nabla_{\bm P}\mathcal{L} &= \hat{\bm \Sigma} -\bm \Lambda = \bm O,\label{dual:cond1}\\
    \nabla_{\bm Q}\mathcal{L} &= \kappa_2 \hat{\bm \Sigma}  - \hat{\bm \mu} \hat{\bm \mu}^\top -\frac{1}{2}(\hat{\bm \mu}\bm \alpha^\top + \bm \alpha\hat{\bm \mu}^\top)-\sum_{\ell\in [L]}\bm B^{(\ell)}=\bm O,\\
    \nabla_{\bm p}\mathcal{L} &=-2\hat{\bm \mu}-\bm \alpha -2\bm \lambda = \bm 0,\\
    \nabla_{\bm q}\mathcal{L} &= -\frac{1}{2}\bm \alpha - \sum_{\ell\in [L]}\bm \beta^{(\ell)} = \bm 0,\\
    \nabla_{r}\mathcal{L} &= 1 - \sum_{\ell\in [L]}\eta^{(\ell)} = 0,\\
    \nabla_{s}\mathcal{L} &= \kappa_1 - \nu = 0.
\end{align}
The following optimality condition should also be satisfied:
\begin{align}
    \nabla_{\bm x}\mathcal{L} &= \frac{1}{\gamma}\bm x - \bm Z\left(\sum_{\ell\in [L]}a^{(\ell)} \bm \beta^{(\ell)} +\pi\bm 1 + \bm \rho \right) = \bm 0.\label{dual:cond6}
\end{align}
According to conditions \eqref{dual:cond1}--\eqref{dual:cond6}, the optimal objective value of Problem \eqref{prob:relaxed_prob} is calculated as 
\begin{equation*}
 -\frac{\gamma}{2} \bm \omega^\top \bm Z^2 \bm \omega - \sum_{\ell\in [L]}\eta^{(\ell)}b^{(\ell)} + \pi,
\end{equation*}
where 
\begin{equation*}
    \bm \omega = \sum_{\ell\in [L]}a^{(\ell)} \bm \beta^{(\ell)} +\pi\bm 1+\bm \rho.
\end{equation*}
Since 
$\bm \omega^\top \bm Z^2\bm \omega = \bm z^\top (\bm \omega\circ \bm \omega)$, the Lagrange dual~\eqref{prob:lagrange_dual} of Problem \eqref{prob:lower_primal} is formulated as
\begin{subequations}\label{prob:lagrange_dual_tmp}
\begin{align}
\maximize_{\bm \omega,\bm \alpha, \bm B, \bm \beta, \bm \eta,\bm  \Lambda,\bm \lambda,\nu,\pi,\bm \rho} &\quad -\frac{\gamma}{2} \bm z^\top (\bm \omega\circ \bm \omega) -\sum_{\ell\in [L]} \eta^{(\ell)} b^{(\ell)} + \pi\\
\subjectto & \quad \bm \omega = \sum_{\ell\in [L]}a^{(\ell)} \bm \beta^{(\ell)} +\pi\bm 1 + \bm \rho,\label{const:rho}\\
&\quad \hat{\bm \Sigma} - \bm \Lambda = \bm O,\label{const:lambda}\\
&\quad \kappa_2 \hat{\bm \Sigma}  - \hat{\bm \mu} \hat{\bm \mu}^\top - \frac{1}{2}(\hat{\bm \mu}\bm \alpha^\top+\bm \alpha\hat{\bm \mu}^\top) -\sum_{\ell\in [L]}\bm B^{(\ell)}=\bm O,\label{const:alpha1}\\
&\quad -2\hat{\bm \mu}-\bm \alpha -2\bm \lambda = \bm 0, \label{const:alpha2}\\
&\quad -\frac{1}{2}\bm \alpha - \sum_{\ell\in [L]}\bm \beta^{(\ell)} = \bm 0,\label{const:alpha3}\\
&\quad  1 - \sum_{\ell\in [L]}\eta^{(\ell)} = 0,\\
&\quad  \kappa_1 - \nu = 0,\label{const:nu}\\
&\quad  \begin{pmatrix}
     \bm B^{(\ell)} &\bm \beta^{(\ell)}\\
     (\bm \beta^{(\ell)})^{\top} & \eta^{(\ell)}
    \end{pmatrix}\succeq \bm O\quad (\ell \in [L]),\\
&\quad \begin{pmatrix}
     \bm \Lambda &\bm \lambda\\
     \bm \lambda^\top & \nu
    \end{pmatrix}\succeq \bm O,\\
&\quad \bm \rho \geq \bm 0.
\end{align}
\end{subequations}
We then delete $\bm \rho$, $\bm \Lambda$ $\bm \alpha$, and $\nu$ from Problem~\eqref{prob:lagrange_dual_tmp} by substituting Eqs.~\eqref{const:rho}, \eqref{const:lambda}, \eqref{const:alpha3}, and \eqref{const:nu} into other constraints. 
We now obtain the desired formulation \eqref{prob:lower_dual}.  

To prove the strong duality~\citep{Alizadeh1995}, 
we next show that both the primal problem~\eqref{prob:lower_primal} and the dual problem~\eqref{prob:lower_dual} are strictly feasible. 
Let us set $\bm x = \bm z/(\bm 1^{\top} \bm z)$, $\bm q = \bm 0$, $r > -\min_{\ell\in [L]} \{b^{(\ell)}\}$, and $s>0$ in the primal problem~\eqref{prob:lower_primal}.
Then there exists $\bm Q$ satisfying 
\begin{equation*}
    \begin{pmatrix}
        \bm Q & a^{(\ell)}\bm Z\bm x/2\\
        (a^{(\ell)}\bm Z\bm x/2)^\top &r+ b^{(\ell)}
        \end{pmatrix} \succ \bm O \quad (\forall \ell \in [L]).
\end{equation*}
We next set $\bm p = - \bm Q \hat{\bm \mu}$, then there exists $\bm P$ satisfying
\begin{equation*}
    \begin{pmatrix}
            \bm P &\bm p\\
            \bm p^\top &s
            \end{pmatrix} \succ \bm O,
\end{equation*}
which implies that the primal problem \eqref{prob:lower_primal} is strictly feasible. 

For the dual problem \eqref{prob:lower_dual}, we set $\bm \lambda = \bm 0$ and $(\bm B,\bm \beta, \bm \eta)$ as
\begin{align*}
    \bm B^{(\ell)} = \frac{1}{L}(\kappa_2 \hat{\bm \Sigma}+\hat{\bm \mu}\hat{\bm \mu}^\top),\quad
    \bm \beta^{(\ell)} =\frac{1}{L}\hat{\bm \mu},\quad
    \eta^{(\ell)} = \frac{1}{L} \quad (\forall \ell\in [L]) 
\end{align*}
such that constraints \eqref{dual_const_eq:1}--\eqref{dual_const_eq:3} are satisfied. 
We next set $\bm \omega > \sum_{\ell\in [L]}a^{(\ell)} \bm \beta^{(\ell)} +\pi\bm 1$. 
As for constraint \eqref{dual_const_sd:1}, we have 
\begin{align}
     \begin{pmatrix}
     \bm B^{(\ell)} &\bm \beta^{(\ell)}\\
     (\bm \beta^{(\ell)})^{\top} & \eta^{(\ell)}
    \end{pmatrix} &= \frac{1}{L}\begin{pmatrix}\kappa_2 \hat{\bm \Sigma}+\hat{\bm \mu}\hat{\bm \mu}^\top & \hat{\bm \mu}\\
                                                \hat{\bm \mu}^\top & 1 
                                \end{pmatrix} \succ \bm O \quad (\forall \ell \in [L]) \label{eq:tmp_mat}
\end{align}
because the Schur complement~\citep{Boyd2009} is given by 
\begin{equation}
\frac{1}{L}\left(\kappa_2 \hat{\bm \Sigma}+\hat{\bm \mu}\hat{\bm \mu}^\top - \hat{\bm \mu}\hat{\bm \mu}^\top\right) = \frac{\kappa_2}{L} \hat{\bm \Sigma} \succ \bm O
\end{equation}
from \Cref{asp:1}.
For constraint \eqref{dual_const_sd:2}, \Cref{asp:1} also ensures that
\begin{align}
        \begin{pmatrix}
        \hat{\bm \Sigma} & \bm 0\\
        \bm 0 & \kappa_1
        \end{pmatrix} \succ \bm O\label{pd},
\end{align}
which shows that the dual problem \eqref{prob:lower_dual} is strictly feasible.

\section{Proof of \Cref{lem:bar_beta_ineq}}\label{sec:proof_bar_beta_ineq}
We assume without loss of generality that $\eta^{(\ell)}>0$ for all $\ell  \in [L]$ and that 
\begin{equation} \label{eq:z_partition}
        \bm z = \begin{pmatrix}
        \bm 1\\
        \bm 0
        \end{pmatrix},~\bm 1 \in \mathbb{R}^k,~\bm 0 \in \mathbb{R}^{N-k}.
\end{equation}
For notational simplicity, we partition vector $\bm v \in \mathbb{R}^N$ and symmetric matrix $\bm M \in \mathcal{S}^{N}$ in  accordance with $\bm z$ as follows:
\begin{align}
\bm v & = \begin{pmatrix}
        \bm v_1\\
        \bm v_2
        \end{pmatrix},~\bm v_1 =\bm v_{\bm z},~\bm v_2=\bm v_{\bm 1-\bm z}, \label{eq:v_partition} \\
\bm M & = \begin{pmatrix}
        \bm M_{11} & \bm M_{12}\\
        \bm M_{21} & \bm M_{22}
        \end{pmatrix},~\bm M_{11} = \bm M_{\bm z,\bm z},~\bm M_{12} = \bm M_{21}^{\top} = \bm M_{\bm z, \bm 1-\bm z},~\bm M_{22}=\bm M_{\bm 1-\bm z,\bm 1-\bm z}. \label{eq:M_partition}
\end{align}

We define $\bm \phi^{(\ell)}\coloneqq \bar{\bm \beta}^{(\ell)}-\eta^{(\ell)}\hat{\bm \mu}$ for $\ell\in [L]$, and due to definition~\eqref{def:beta}, we have
\begin{align}
    \bm \phi^{(\ell)}_1 &=\bm \beta_1^{(\ell)}-\eta^{(\ell)}\hat{\bm \mu}_1,\label{eq:phi1}\\
    \bm \phi^{(\ell)}_2 &=\left(\hat{\bm \Sigma}_{21 }(\hat{\bm \Sigma}_{11})^{-1}(\bm \beta_{1} ^{(\ell)}-\eta^{(\ell)}\hat{\bm \mu}_{1} )+\eta^{(\ell)}\hat{\bm \mu}_{2}\right)-\eta^{(\ell)}\hat{\bm \mu}_{2} = \bm \Psi \bm \phi^{(\ell)}_1,
\end{align}
where 
\begin{equation} \label{eq:Psi}
\bm \Psi\coloneqq\hat{\bm \Sigma}_{21 }(\hat{\bm \Sigma}_{11})^{-1}.
\end{equation}
Substituting them into the right side of Eq.~\eqref{eq:matrix_ineq}, we obtain
\begin{align}
    & \sum_{\ell\in [L]}\frac{1}{\eta^{(\ell)}}(\bar{\bm \beta}^{(\ell)}-\eta^{(\ell)}\hat{\bm \mu})(\bar{\bm \beta}^{(\ell)}-\eta^{(\ell)}\hat{\bm \mu})^\top \notag \\
    =&\sum_{\ell\in [L]}\frac{1}{\eta^{(\ell)}}\begin{pmatrix}
    \bm \phi^{(\ell)}_1(\bm \phi^{(\ell)}_1)^\top  & \bm \phi^{(\ell)}_1(\bm \phi^{(\ell)}_2)^\top  \\
     \bm \phi^{(\ell)}_2(\bm \phi^{(\ell)}_1)^\top  & \bm \phi^{(\ell)}_2(\bm \phi^{(\ell)}_2)^\top  
    \end{pmatrix}
    =\begin{pmatrix}
    \bm \Phi &\bm \Phi \bm \Psi^\top\\
    \bm \Psi \bm \Phi & \bm \Psi \bm \Phi \bm \Psi^\top
    \end{pmatrix}\label{eq:phi_ineq},
\end{align}
where
\begin{equation}\label{eq:Phi}
\bm \Phi\coloneqq\sum_{\ell\in [L]}\frac{1}{\eta^{(\ell)}}\bm \phi^{(\ell)}_1(\bm \phi^{(\ell)}_1)^\top.
\end{equation}

From constraint~\eqref{lower_dual_reduced_const_sd:1}, we can use the Schur complement property~\citep{Boyd2009}: 
\begin{equation}\label{eq:Schur1}
    \begin{pmatrix}
     \bm B^{(\ell)}_{11} &\bm \beta^{(\ell)}_{1}\\
     (\bm \beta_{1}^{(\ell)})^{\top} & \eta^{(\ell)}
    \end{pmatrix}\succeq \bm O~\Rightarrow~
    \bm B^{(\ell)}_{11}\succeq \frac{1}{\eta^{(\ell)}}\bm \beta^{(\ell)}_{1} (\bm \beta^{(\ell)}_{1})^\top \quad (\forall \ell \in [L]).
\end{equation}
It then follows 
that 
\begin{align}
    \bm \Phi 
    &= \sum_{\ell\in [L]}\frac{1}{\eta^{(\ell)}}\bm \beta_1^{(\ell)}(\bm \beta_1^{(\ell)})^\top - \left(\sum_{\ell\in [L]}\bm \beta_1^{(\ell)}\right)\hat{\bm \mu}_1^\top - \hat{\bm \mu}_1\left(\sum_{\ell\in [L]}\bm \beta_1^{(\ell)}\right)^\top  + \hat{\bm \mu}_1\hat{\bm \mu}_1^\top \quad \because \mathrm{Eqs.~ \eqref{lower_dual_reduced_const_eq:3}, \eqref{eq:phi1}, \eqref{eq:Phi}} \notag\\ 
    &\preceq \sum_{\ell\in [L]}\bm B^{(\ell)}_{11} + \hat{\bm \mu}_1\hat{\bm \mu}_1^\top -\hat{\bm \mu}_1\left(\sum_{\ell\in [L]}\bm \beta_1^{(\ell)}\right)^\top - \left(\sum_{\ell\in [L]}\bm \beta_1^{(\ell)}\right)\hat{\bm \mu}_1^\top \quad \because \mathrm{Eq.~\eqref{eq:Schur1}} \notag \\ 
    &= \kappa_2\hat{\bm \Sigma}_{11}. \quad \because \mathrm{Eq.~\eqref{lower_dual_reduced_const_eq:1}} \label{eq:phi_eq}
\end{align}

From \Cref{asp:1}, 
we can also use the Schur complement property~\citep{Boyd2009}: 
\begin{equation}\label{eq:Schur2}
    \begin{pmatrix}
    \hat{\bm \Sigma}_{11} &\hat{\bm \Sigma}_{21}^\top\\
    \hat{\bm \Sigma}_{21 }&\hat{\bm \Sigma}_{22}
         \end{pmatrix} \succ \bm O ~\Rightarrow~ 
\hat{\bm \Sigma}_{22}-\hat{\bm \Sigma}_{21 }(\hat{\bm \Sigma}_{11})^{-1}\hat{\bm \Sigma}_{21}^\top \succ \bm O.
\end{equation}
We now obtain the desired result:
\begin{align}
       \sum_{\ell\in [L]}\frac{1}{\eta^{(\ell)}}(\bar{\bm \beta}^{(\ell)}-\eta^{(\ell)}\hat{\bm \mu})(\bar{\bm \beta}^{(\ell)}-\eta^{(\ell)}\hat{\bm \mu})^\top 
    & = \begin{pmatrix}
    \bm \Phi &\bm \Phi \bm \Psi^\top\\
    \bm \Psi \bm \Phi & \bm \Psi \bm \Phi \bm \Psi^\top
    \end{pmatrix} \quad \because \mathrm{Eq.~\eqref{eq:phi_ineq}}\notag \\
    & =
        \begin{pmatrix}
        \bm I\\
        \bm \Psi
        \end{pmatrix}
        \bm \Phi
        \begin{pmatrix}
        \bm I &\bm \Psi^\top
        \end{pmatrix}\notag
        \\
    & \preceq 
        \begin{pmatrix}
        \bm I\\
        \bm \Psi
        \end{pmatrix}
        \kappa_2\hat{\bm \Sigma}_{11}
        \begin{pmatrix}
        \bm I &\bm \Psi^\top
        \end{pmatrix} \quad \because \mathrm{Eq.~\eqref{eq:phi_eq}}  \notag
        \\
    & = \kappa_2\begin{pmatrix}
    \hat{\bm \Sigma}_{11} &\hat{\bm \Sigma}_{21}^\top\\
    \hat{\bm \Sigma}_{21 }&\hat{\bm \Sigma}_{21 }(\hat{\bm \Sigma}_{11})^{-1}\hat{\bm \Sigma}_{21}^\top
    \end{pmatrix} \quad \because \mathrm{Eq.~\eqref{eq:Psi}} \notag \\
    & \preceq \kappa_2\begin{pmatrix}
    \hat{\bm \Sigma}_{11} &\hat{\bm \Sigma}_{21}^\top\\
    \hat{\bm \Sigma}_{21 }&\hat{\bm \Sigma}_{22}
    \end{pmatrix}. \quad \because \mathrm{Eq.~\eqref{eq:Schur2}} \notag
\end{align}

\section{Proof of \Cref{lem:B_completion}}\label{sec:proof_B_completion}
For $\ell \in [L]$ such that $\eta^{(\ell)} = 0$, we have $\bm \beta_{\bm z}^{(\ell)} = \bm 0$ from the constraint \eqref{lower_dual_reduced_const_sd:1}, thus $\bar{\bm \beta}^{(\ell)}= \bm 0$ from the definition \eqref{def:beta}. 
In this case, the solution $(\bar{\bm B}^{(\ell)},\bar{\bm \beta}^{(\ell)},\eta^{(\ell)})=(\bm O,\bm 0, 0)$ satisfies the constraint \eqref{dual_const_sd:1} and can be left out of Eq.~\eqref{dual_const_eq:1}. 
Therefore, we can assume without loss of generality that $\eta^{(\ell)}>0$ for all $\ell \in [L]$ in Eqs.~\eqref{dual_const_eq:1} and \eqref{dual_const_sd:1}. 

From Eqs.~\eqref{lower_dual_reduced_const_eq:3}~and~\eqref{eq:matrix_ineq}, we can see that 
\begin{equation}\label{eq:tilde_B_sd1}
    \kappa_2\hat{\bm \Sigma} - \hat{\bm \mu}\hat{\bm \mu}^\top +\hat{\bm \mu}\left(\sum_{\ell\in [L]}\bar{\bm \beta}^{(\ell)}\right)^\top +  \left(\sum_{\ell\in [L]}\bar{\bm \beta}^{(\ell)}\right)\hat{\bm \mu}^\top \succeq \sum_{\ell\in [L]}\widetilde{\bm B}^{(\ell)},
\end{equation}
where
\[
\widetilde{\bm B}^{(\ell)} \coloneqq \frac{1}{\eta^{(\ell)}}\bar{\bm \beta}^{(\ell)}(\bar{\bm \beta}^{(\ell)})^\top \quad (\forall \ell \in [L]). 
\]
Then we have
\begin{equation}\label{eq:tilde_B_sd2}
    \begin{pmatrix}
    \widetilde{\bm B}^{(\ell)} &\bar{\bm \beta}^{(\ell)}\\
    (\bar{\bm \beta}^{(\ell)})^\top & \eta^{(\ell)}
    \end{pmatrix}
    = \frac{1}{\eta^{(\ell)}}\begin{pmatrix} \bar{\bm \beta}^{(\ell)} \\ \eta^{(\ell)}
    \end{pmatrix} \begin{pmatrix} \bar{\bm \beta}^{(\ell)} \\ \eta^{(\ell)}
    \end{pmatrix}^{\top}
    \succeq \bm O\quad(\forall \ell \in [L]). 
\end{equation}
From Eqs.~\eqref{eq:tilde_B_sd1} and \eqref{eq:tilde_B_sd2}, we can see that $\bar{\bm B}$ satisfying Eqs.~\eqref{dual_const_eq:1} and \eqref{dual_const_sd:1} is obtained by adding appropriate positive semidefinite matrices to $\widetilde{\bm B}^{(\ell)}$ for $\ell \in [L]$. 

\section{Proof of \Cref{thm:reduction}}\label{sec:proof_reduction}
As in the proof of \ref{sec:proof_bar_beta_ineq}, we use the notations~\eqref{eq:v_partition} and \eqref{eq:M_partition} under assumption~\eqref{eq:z_partition}. 
Suppose that $(\bm \omega, \bm B,\bm \beta,\bm \eta,\bm \lambda, \pi)$ is an optimal solution to Problem \eqref{prob:lower_dual}. 
Then $(\bm \omega_1, \bm B_{11},\bm \beta_1,\bm \eta,\bm \lambda_1, \pi)$ is a feasible solution to Problem \eqref{prob:lower_dual_reduced}, and its objective value is equal to $f(\bm z)$ due to Eq.~\eqref{eq:l2reg}.
This proves that $f(\bm z)\le f'(\bm z)$. 

Next, we show that $f(\bm z)\ge f'(\bm z)$. 
Let $(\bm \omega_1, \bm B_{11},\bm \beta_1,\bm \eta,\bm \lambda_1, \pi)$ be an optimal solution to Problem \eqref{prob:lower_dual_reduced}. 
From definition~\eqref{eq:def_bar_omega}, the constraint \eqref{dual_const_ineq:1} is satisfied by $(\bar{\bm \omega},\bar{\bm \beta}, \pi)$. 
The constraint \eqref{dual_const_eq:2} is satisfied by $(\bar{\bm \beta},\bar{\bm \lambda})$ as 
\begin{align*}
    \sum_{\ell\in [L]} \bar{\bm \beta}^{(\ell)} 
    &= \begin{pmatrix}
    \sum_{\ell\in [L]}\bm \beta_1^{(\ell)}\\
    \hat{\bm \Sigma}_{21}(\hat{\bm \Sigma}_{11})^{-1}\left(\sum_{\ell\in [L]}\bm \beta_1^{(\ell)}-\hat{\bm \mu}_1\right) + \hat{\bm \mu}_2
    \end{pmatrix} \quad \because \mathrm{Eqs.~\eqref{lower_dual_reduced_const_eq:3},\eqref{def:beta}}\\
    &= \begin{pmatrix}
    \bm \lambda_1 + \hat{\bm \mu}_1\\
    \hat{\bm \Sigma}_{21}(\hat{\bm \Sigma}_{11})^{-1}\bm \lambda_1+\hat{\bm \mu}_2
    \end{pmatrix} \quad \because \mathrm{Eq.~\eqref{lower_dual_reduced_const_eq:2}}\\
    &=\bar{\bm \lambda} + \hat{\bm \mu}. \quad \because \mathrm{Eq.~\eqref{eq:def_bar_lambda}}
\end{align*}
Since the constraint \eqref{lower_dual_reduced_const_sd:2} is satisfied, we can use the following Schur complement property (see, e.g., Section A.5.5 in~\citet{Boyd2009}):
\begin{equation}
    \begin{pmatrix}
     \hat{\bm \Sigma}_{11} &\bm \lambda_{1}\\
     \bm \lambda_{1}^\top & \kappa_1
    \end{pmatrix}\succeq \bm O ~\Rightarrow~ \kappa_1- \bm \lambda_1^\top\hat{\bm \Sigma}_{11}^{-1} \bm \lambda_1 \geq 0.\label{eq:Schur}
\end{equation}
From~Eq.~\eqref{eq:def_bar_lambda}, we have
\begin{equation*}
    \begin{pmatrix}
    \hat{\bm \Sigma}& \bar{\bm \lambda}\\
    \bar{\bm \lambda}^\top & \kappa_1
    \end{pmatrix} = 
    \begin{pmatrix}
    \hat{\bm \Sigma}_{11}&\hat{\bm \Sigma}_{12} &\bar{\bm \lambda}_{1}\\
    \hat{\bm \Sigma}_{21}&\hat{\bm \Sigma}_{22} &\bar{\bm \lambda}_{2}\\
    \bar{\bm \lambda}_{1}^\top &\bar{\bm \lambda}_{2}^\top &\kappa_1
    \end{pmatrix} = 
    \begin{pmatrix}
    \hat{\bm \Sigma}_{11}&\hat{\bm \Sigma}_{12} &\bm \lambda_{1}\\
    \hat{\bm \Sigma}_{21}&\hat{\bm \Sigma}_{22} &\hat{\bm \Sigma}_{21}\hat{\bm \Sigma}_{11}^{-1} \bm \lambda_1 \\
    \bm \lambda_{1}^\top &(\hat{\bm \Sigma}_{21}\hat{\bm \Sigma}_{11}^{-1} \bm \lambda_1 )^\top &\kappa_1
    \end{pmatrix}, 
\end{equation*}
which is positive semidefinite because
\begin{align*}
    &\begin{pmatrix}
    \bm I &\bm O &\bm 0\\
    \bm O &\bm I &\bm 0\\
    -\bm \lambda_1^\top \hat{\bm \Sigma}^{-1}_{11} &\bm 0^{\top} &1
    \end{pmatrix}
    \begin{pmatrix}
    \hat{\bm \Sigma}_{11}&\hat{\bm \Sigma}_{12} &\bm \lambda_{1}\\
    \hat{\bm \Sigma}_{21}&\hat{\bm \Sigma}_{22} &\hat{\bm \Sigma}_{21}\hat{\bm \Sigma}_{11}^{-1} \bm \lambda_1 \\
    \bm \lambda_{1}^\top &(\hat{\bm \Sigma}_{21}\hat{\bm \Sigma}_{11}^{-1} \bm \lambda_1 )^\top &\kappa_1
    \end{pmatrix}
    \begin{pmatrix}
    \bm I &\bm O &-\hat{\bm \Sigma}^{-1}_{11} \bm \lambda_1\\
    \bm O &\bm I &\bm 0\\
    \bm 0^{\top} &\bm 0^{\top} &1
    \end{pmatrix}
   \\
    = & \begin{pmatrix}
    \hat{\bm \Sigma}_{11} & \hat{\bm \Sigma}_{12}  & \bm 0\\
    \hat{\bm \Sigma}_{21}  & \hat{\bm \Sigma}_{22} &\bm 0\\
    \bm 0^{\top} &\bm 0^{\top} & \kappa_1- \bm \lambda_1^\top\hat{\bm \Sigma}_{11}^{-1}\bm \lambda_1
    \end{pmatrix} \succeq \bm O. \quad \because \mathrm{\Cref{asp:1},~Eq.~\eqref{eq:Schur}}
\end{align*}
Therefore, the constraint \eqref{dual_const_sd:2} is satisfied by $\bar{\bm \lambda}$. 

From \Cref{lem:B_completion}, there exists $\bar{\bm B}$ satisfying constraints~\eqref{dual_const_eq:1}~and~\eqref{dual_const_sd:1}.
Therefore, $(\bar{\bm \omega}, \bar{\bm B}, \bar{\bm \beta},\bm \eta,\bar{\bm \lambda}, \pi)$ is a feasible solution to Problem \eqref{prob:lower_dual}, and its objective value is equal to $f'(\bm z)$ due to Eq.~\eqref{eq:l2reg}. 
This implies that $f(\bm z) \geq f'(\bm z)$. 
As a result, we have $f(\bm z) = f'(\bm z)$; thus, $(\bar{\bm \omega}, \bar{\bm B}, \bar{\bm \beta},\bm \eta,\bar{\bm \lambda}, \pi)$ is an optimal solution to Problem \eqref{prob:lower_dual}.

\section{Illustration of utility function}\label{apd:utility}
\Cref{fig:utility function} illustrates the piecewise-linear approximation of the exponential utility function~\eqref{eq:exp_utility} with $L=3$.

\begin{figure}[ht]
    \centering
\begin{tikzpicture}[scale=3.5]
   \fill [black] (0,0) circle [radius=0.015];
\fill [black] (0.5,{1-1/exp(0.5)}) circle [radius=0.015];
   \fill [black] (1,{1-1/exp(1)}) circle [radius=0.015];
   \path[draw,->,>=latex] (0, 0) -- (1.5,0) node[right] {\small $y$};
   \path[draw,->,>=latex] (0, 0) -- (0,1) node[left=2mm] {\small $u$} ;
   \path (0,0) node[below left] {$\mathrm{O}$};
   \path[draw,domain=0:1.5,dashed] plot (\x, {1-exp(-\x/1)}) node[right] {\small $\tilde{u}(y)$};
   \path[draw,domain=0:1] plot (\x, {\x}) node[above]{\small $a^{(1)}y+b^{(1)}$};
   \path[draw,domain=0:1.5] plot (\x, {\x/exp(1/2)+1-3*exp(-1/2)/2})node[above right]{\small $a^{(2)}y+b^{(2)}$};
   \path[draw,domain=0:1.5] plot (\x, {\x/exp(1)+1-2/exp(1)})node[above right]{\small $a^{(3)}y+b^{(3)}$};
   \path[draw,dotted] (1,{1-1/exp(1)})--(0, {1-1/exp(1)}) node[left]{$\small  \frac{\mu_{\max}(1-\exp(-10))}{10}$};
   \path[draw,dotted] (0.5,{1-1/exp(0.5)})--(0, {1-1/exp(0.5)}) node[left]{$\small  \frac{\mu_{\max}(1-\exp(-5))}{10}$};
   \path[draw,dotted] (1/2,{1-1/exp(1/2)})--(1/2, 0) node[below]{\small $\frac{\mu_{\max}}{2}$};
   \path[draw,dotted] (1,{1-1/exp(1)})--(1, 0) node[below]{\small $\mu_{\max}$};
\end{tikzpicture}
       \caption{Piecewise-linear approximation of utility function~\eqref{eq:exp_utility} with $L=3$}
    \label{fig:utility function}
\end{figure}
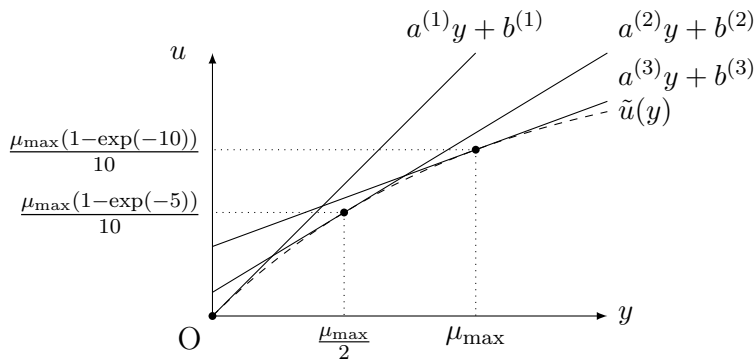

\section{Additional results of computational performance}\label{apd:c}

\subsection*{Results for different values of cardinality parameter $k$.}
\Cref{tab:result_k1,tab:result_k2} give the numerical results of each method for the cardinality parameter $k\in \{5,10,15,20,25\}$. 
We set $\gamma=10/\sqrt{N}$ for the $\ell_2$-regularization term and $(\kappa_1,\kappa_2)= (1,4)$ for the ambiguity set. 
We can see that CPA${}_+$ consistently outperformed the other methods regardless of the cardinality parameter $k$.

\begin{table}[p]
    \centering
        \caption{Numerical results for the \texttt{ind49}, \texttt{port2}, and \texttt{sbm100} datasets with $\gamma=10/\sqrt{N}$ and $(\kappa_1,\kappa_2)= (1,4)$ for $k\in \{5,10,15, 20, 25\}$}
        \label{tab:result_k1}
     \footnotesize
      \renewcommand{\arraystretch}{0.90}
\begin{tabular}{lrrlrrrrrr}
\toprule
Data &$N$& $k$ 	&Method	 &Obj	 &Gap(\%)	&Time 	&\#Cuts 	&\#Nodes\\ \hline
\texttt{ind49}	&49	&5	&SCIP	&3.108	&0.0	&1365.8	&---	&61\\
	&	&	&CPA	&3.108	&24.1	&$>$3600	&$>$167	&$>$785\\
	&	&	&CPA${}_+$	&3.108	&0.0	&\textbf{40.4}	&211	&1000\\ \cline{3-9} 
	&	&10	&SCIP	&3.034	&0.0	&663.2	&---	&24\\
	&	&	&CPA	&3.034	&0.0	&469.6	&22	&20\\
	&	&	&CPA${}_+$	&3.034	&0.0	&\textbf{12.0}	&22	&20\\ \cline{3-9} 
	&	&15	&SCIP	&3.034	&0.0	&1049.2	&---	&49\\
	&	&	&CPA	&3.033	&0.0	&127.7	&6	&1\\
	&	&	&CPA${}_+$	&3.033	&0.0	&\textbf{6.6}	&6	&1\\ \cline{3-9} 
	&	&20	&SCIP	&3.033	&0.0	&502.8	&---	&27\\
	&	&	&CPA	&3.033	&0.0	&107.9	&5	&1\\
	&	&	&CPA${}_+$	&3.033	&0.0	&\textbf{10.3}	&5	&1\\ \cline{3-9} 
	&	&25	&SCIP	&3.033	&0.0	&542.3	&---	&35\\
	&	&	&CPA	&3.033	&0.0	&106.6	&5	&1\\
	&	&	&CPA${}_+$	&3.033	&0.0	&\textbf{16.9}	&5	&1\\ \hline 
 \texttt{port2}	&89	&5	&SCIP	&2.217	&27.5	&$>$3600	&---&$>$14\\
	&	&	&CPA	&2.181	&100.0	&$>$3600	&$>$28	&$>$110\\
	&	&	&CPA${}_+$	&\textbf{2.027}	&71.7	&$>$3600	&$>$16645	&$>$155410\\ \cline{3-9} 
	&	&10	&SCIP	&1.865	&13.8	&$>$3600	&---&$>$17\\
	&	&	&CPA	&1.924	&100.0	&$>$3600	&$>$28	&$>$91\\
	&	&	&CPA${}_+$	&\textbf{1.726}	&55.2	&$>$3600	&$>$6380	&$>$110529\\ \cline{3-9} 
	&	&15	&SCIP	&1.859	&13.5	&$>$3600	&---&$>$14\\
	&	&	&CPA	&2.009	&96.4	&$>$3600	&$>$24	&$>$51\\
	&	&	&CPA${}_+$	&\textbf{1.637}	&19.4	&$>$3600	&$>$3027	&$>$51474\\ \cline{3-9} 
	&	&20	&SCIP	&1.797	&10.6	&$>$3600	&---&$>$13\\
	&	&	&CPA	&1.866	&100.0	&$>$3600	&$>$24	&$>$151\\
	&	&	&CPA${}_+$	&\textbf{1.611}	&1.9	&$>$3600	&$>$1665	&$>$22962\\ \cline{3-9} 
	&	&25	&SCIP	&1.783	&9.9	&$>$3600	&---&$>$18\\
	&	&	&CPA	&\textbf{1.607}	&0.0	&3281.5	&21	&6\\
	&	&	&CPA${}_+$	&\textbf{1.607}	&0.0	&\textbf{77.3}	&22	&12\\ \hline 
 \texttt{sbm100}	&100	&5	&SCIP	&4.493	&12.4	&$>$3600	&---&$>$5\\
	&	&	&CPA	&3.940	&0.3	&$>$3600	&$>$12	&$>$1\\
	&	&	&CPA${}_+$	&3.940	&0.0	&\textbf{2.7}	&13	&8\\ \cline{3-9} 
	&	&10	&SCIP	&3.970	&0.9	&$>$3600	&---&$>$9\\
	&	&	&CPA	&3.935	&0.0	&1491.5	&5	&1\\
	&	&	&CPA${}_+$	&3.935	&0.0	&\textbf{2.6}	&5	&1\\ \cline{3-9} 
	&	&15	&SCIP	&3.972	&0.9	&$>$3600	&---&$>$5\\
	&	&	&CPA	&3.935	&0.0	&1810.1	&6	&1\\
	&	&	&CPA${}_+$	&3.935	&0.0	&\textbf{5.2}	&5	&1\\ \cline{3-9} 
	&	&20	&SCIP	&3.952	&0.2	&$>$3600	&---&$>$8\\
	&	&	&CPA	&3.935	&0.0	&1626.7	&5	&1\\
	&	&	&CPA${}_+$	&3.935	&0.0	&\textbf{10.5}	&5	&1\\ \cline{3-9} 
	&	&25	&SCIP	&4.595	&14.3	&$>$3600	&---&$>$5\\
	&	&	&CPA	&3.935	&0.0	&1558.0	&5	&1\\
	&	&	&CPA${}_+$	&3.935	&0.0	&\textbf{16.8}	&5	&1\\ 
\bottomrule
\end{tabular}
\end{table}
\begin{table}[p]
    \centering
        \caption{Numerical results for the \texttt{port5} and \texttt{s\&p500} datasets with $\gamma=10/\sqrt{N}$ and $(\kappa_1,\kappa_2)= (1,4)$ for $k\in \{5,10,15,20,25\}$}
        \label{tab:result_k2}
     \footnotesize
      \renewcommand{\arraystretch}{0.90}
\begin{tabular}{lrrlrrrrrr}
\toprule
Data &$N$& $k$ 	&Method	 &Obj	 &Gap(\%)	&Time 	&\#Cuts 	&\#Nodes\\ \hline
 \texttt{port5}	&225	&5	&SCIP	&$\infty$	&100.0	&$>$3600	&---&$>$1\\
    &   &   &CPA  & OM &--- &---&---&---\\
	&	&	&CPA${}_+$	&2.812	&0.0	&\textbf{164.6}	&840	&9285\\ \cline{3-9} 
	&	&10	&SCIP	&$\infty$	&100.0	&$>$3600	&---&$>$1\\
    &   &   &CPA  & OM &--- &---&---&---\\
    &	&	&CPA${}_+$	&2.687	&0.0	&\textbf{97.8}	&178	&1751\\ \cline{3-9} 
	&	&15	&SCIP	&$\infty$	&100.0	&$>$3600	&---&$>$1\\
	    &   &   &CPA  & OM &--- &---&---&---\\
	&	&	&CPA${}_+$	&2.677	&0.0	&\textbf{21.7}	&19	&34\\ \cline{3-9} 
	&	&20	&SCIP	&$\infty$	&100.0	&$>$3600	&---&$>$1\\
	    &   &   &CPA  & OM &--- &---&---&---\\
	    &	&	&CPA${}_+$	&2.677	&0.0	&\textbf{12.2}	&6	&1\\ \cline{3-9} 
	&	&25	&SCIP	&$\infty$	&100.0	&$>$3600	&---&$>$1\\
	    &   &   &CPA  & OM &--- &---&---&---\\
	    &	&	&CPA${}_+$	&2.677	&0.0	&\textbf{21.1}	&6	&1\\ \hline 
 \texttt{s\&p500}	&468	&5	&SCIP	&$\infty$	&---	&$>$3600	&---&$>$1\\
	    &   &   &CPA  & OM &--- &---&---&---\\
	    &	&	&CPA${}_+$	&\textbf{1.058}	&19.9	&$>$3600	&$>$8747	&$>$382504\\ \cline{3-9} 
	&	&10	&SCIP	&$\infty$	&---	&$>$3600	&---&$>$1\\
	    &   &   &CPA  & OM &--- &---&---&---\\
	&	&	&CPA${}_+$	&\textbf{0.888}	&15.3	&$>$3600	&$>$4962	&$>$398075\\ \cline{3-9} 
	&	&15	&SCIP	&$\infty$	&100.0	&$>$3600	&---&$>$1\\
	    &   &   &CPA  & OM &--- &---&---&---\\
	    &	&	&CPA${}_+$	&\textbf{0.833}	&10.0	&$>$3600	&$>$2948	&$>$157613\\ \cline{3-9} 
	&	&20	&SCIP	&$\infty$	&100.0	&$>$3600	&---&$>$1\\
	    &   &   &CPA  & OM &--- &---&---&---\\
	    &	&	&CPA${}_+$	&\textbf{0.809}	&7.2	&$>$3600	&$>$1702	&$>$83635\\ \cline{3-9} 
	&	&25	&SCIP	&$\infty$	&100.0	&$>$3600	&---&$>$1\\
	    &   &   &CPA  & OM &--- &---&---&---\\
	    &	&	&CPA${}_+$	&\textbf{0.789}	&5.5	&$>$3600	&$>$1079	&$>$48047\\ 
\bottomrule
\end{tabular}
\end{table}

\subsection*{Results for different values of ambiguity parameters $(\kappa_1,\kappa_2)$}
\Cref{tab:result_kappa} gives the numerical results of each method for pairs of ambiguity parameters $(\kappa_1,\kappa_2)\in \{(0.5,2), (1,4), (2,8)\}$. 
We set $k=10$ for the cardinality constraint and $\gamma=10/\sqrt{N}$ for the $\ell_2$-reguralization term.

\Cref{tab:result_kappa} shows that CPA${}_+$ was faster than the other methods when it completed the computations within the time limit regardless of the values of the ambiguity parameters $(\kappa_1,\kappa_2)$. 
We note that in the case of the \texttt{ind49} dataset with $(\kappa_1, \kappa_2) = (2,8)$, SCIP incorrectly determined that the problem instance was infeasible due to numerical instability.
For the \texttt{port2} dataset, all the three methods failed to complete computations within the time limit for all $(\kappa_1,\kappa_2)\in \{(0.5,2), (1,4), (2,8)\}$, but CPA${}_+$ found solutions of better quality than the other methods regardless of the values of $(\kappa_1,\kappa_2)$.  

The computation time of CPA${}_+$ was stable against the change in $(\kappa_1,\kappa_2)$. 
For the \texttt{port5} dataset, CPA${}_+$ finished the computations in 104.5 s, 97.8 s, and 125.7 s for $(\kappa_1,\kappa_2)\in \{(0.5,2), (1,4), (2,8)\}$, respectively. 
In addition, for the \texttt{s\&p500} dataset, optimality gaps given from CPA${}_+$ were 12.8\%, 15.3\%, and 18.3\% for $(\kappa_1,\kappa_2)\in \{(0.5,2), (1,4), (2,8)\}$, respectively. 
These results suggest that the computational performance of CPA${}_+$ is not greatly affected by the ambiguity parameters $(\kappa_1,\kappa_2)$.

\begin{table}[p]
    \centering
    \caption{Numerical results with $k=10$ and $\gamma=10/\sqrt{N}$ for\\ $(\kappa_1,\kappa_2)\in \{(0.5,2),(1,4),(2,8)\}$}
         \label{tab:result_kappa}
     \footnotesize
      \renewcommand{\arraystretch}{0.90}
\begin{threeparttable}
\begin{tabular}{lrrlrrrrrr}
\toprule
Data &$N$& $(\kappa_1,\kappa_2)$ 	&Method	 &Obj	 &Gap(\%)	&Time 	&\#Cuts 	&\#Nodes\\ 
\hline
\texttt{ind49}	&49	&(0.5,2)	&SCIP	&1.946	&0.0	&704.0	&---	&23\\
	&	&	&CPA	&1.946	&0.0	&493.0	&22	&20\\
	&	&	&CPA${}_+$	&1.946	&0.0	&\textbf{11.8}	&22	&20\\ \cline{3-9} 
	&	&(1,4)	&SCIP	&3.034	&0.0	&663.2	&---	&24\\
	&	&	&CPA	&3.034	&0.0	&469.6	&22	&20\\
	&	&	&CPA${}_+$	&3.034	&0.0	&\textbf{12.0}	&22	&20\\ \cline{3-9} 
	&	&(2,8)	&SCIP	&$\infty$\tnote{$\dagger$}	&0.0	&5.3	&---	&1\\
	&	&	&CPA	&4.587	&0.0	&386.0	&18	&9\\
	&	&	&CPA${}_+$	&4.588	&0.0	&\textbf{9.9}	&18	&8\\ \hline 
 \texttt{port2}	&89	&(0.5,2)	&SCIP	&1.428	&24.0	&$>$3600	&---&$>$15\\
	&	&	&CPA	&1.336	&134.4	&$>$3600	&$>$30	&$>$51\\
	&	&	&CPA${}_+$	&\textbf{1.155}	&37.0	&$>$3600	&$>$6546	&$>$118147\\ \cline{3-9} 
	&	&(1,4)	&SCIP	&1.865	&13.8	&$>$3600	&---&$>$17\\
	&	&	&CPA	&1.924	&160.2	&$>$3600	&$>$28	&$>$91\\
	&	&	&CPA${}_+$	&\textbf{1.726}	&55.2	&$>$3600	&$>$6380	&$>$110529\\ \cline{3-9} 
	&	&(2,8)	&SCIP	&2.856	&18.0	&$>$3600	&---&$>$15\\
	&	&	&CPA	&2.766	&242.4	&$>$3600	&$>$28	&$>$71\\
	&	&	&CPA${}_+$	&\textbf{2.485}	&69.4	&$>$3600	&$>$6204	&$>$93726\\ \hline 
 \texttt{sbm100}	&100	&(0.5,2)	&SCIP	&2.612	&1.4	&$>$3600	&---&$>$7\\
	&	&	&CPA	&2.575	&0.0	&1549.4	&5	&1\\
	&	&	&CPA${}_+$	&2.575	&0.0	&\textbf{2.4}	&5	&1\\ \cline{3-9} 
	&	&(1,4)	&SCIP	&3.970	&0.9	&$>$3600	&---&$>$9\\
	&	&	&CPA	&3.935	&0.0	&1491.5	&5	&1\\
	&	&	&CPA${}_+$	&3.935	&0.0	&\textbf{2.6}	&5	&1\\ \cline{3-9} 
	&	&(2,8)	&SCIP	&5.910	&0.6	&$>$3600	&---&$>$5\\
	&	&	&CPA	&5.874	&0.0	&2077.2	&6	&1\\
	    &	&	&CPA${}_+$	&5.874	&0.0	&\textbf{2.5}	&5	&1\\ \hline 
 \texttt{port5}	&225	&(0.5,2)	&SCIP	&$\infty$	&100.0	&$>$3600	&---&$>$1\\
		    &   &   &CPA  & OM &--- &---&---&---\\
		    &	&	&CPA${}_+$	&1.915	&0.0	&\textbf{104.5}	&195	&2175\\ \cline{3-9} 
	&	&(1,4)	&SCIP	&$\infty$	&100.0	&$>$3600	&---&$>$1\\
		    &   &   &CPA  & OM &--- &---&---&---\\
		    &	&	&CPA${}_+$	&2.687	&0.0	&\textbf{97.8}	&178	&1751\\ \cline{3-9} 
	&	&(2,8)	&SCIP	&$\infty$	&100.0	&$>$3600	&---&$>$1\\
		    &   &   &CPA  & OM &--- &---&---&---\\
		    &	&	&CPA${}_+$	&3.776	&0.0	&\textbf{125.7}	&215	&2782\\ \hline 
 \texttt{s\&p500}	&468	&(0.5,2)	&SCIP	&$\infty$	&---	&$>$3600	&---&$>$1\\
		    &   &   &CPA  & OM &--- &---&---&---\\
		    &	&	&CPA${}_+$	&\textbf{0.650}	&12.8	&$>$3600	&$>$4978	&$>$396321\\ \cline{3-9} 
	&	&(1,4)	&SCIP	&$\infty$	&---	&$>$3600	&---&$>$1\\
		    &   &   &CPA  & OM &--- &---&---&---\\
		    &	&	&CPA${}_+$	&\textbf{0.888}	&15.3	&$>$3600	&$>$4962	&$>$398075\\ \cline{3-9} 
	&	&(2,8)	&SCIP	&$\infty$	&---	&$>$3600	&---&$>$1\\
		    &   &   &CPA  & OM &--- &---&---&---\\
		    &	&	&CPA${}_+$	&\textbf{1.231}	&18.3	&$>$3600	&$>$5285	&$>$304218\\ 
	\bottomrule
\end{tabular}
	 \begin{tablenotes}\footnotesize
    \item[$\dagger$] SCIP incorrectly determined that the problem instance was infeasible. 
    \end{tablenotes}
\end{threeparttable}
\end{table}

\subsection*{Results for different values of $\ell_2$-regularization parameter $\gamma$}
\Cref{tab:result_gamma} gives the numerical results of each method for the $\ell_2$-regularization parameter $\gamma\in \{1/\sqrt{N}, 10/\sqrt{N},100/\sqrt{N}\}$. 
We set $k=10$ for the cardinality constraint and $(\kappa_1,\kappa_2)=(1,4)$ for the ambiguity set.

\Cref{tab:result_gamma} shows that CPA${}_+$ performed very well compared with the other methods.
CPA${}_+$ was always faster than the other methods for the \texttt{ind49}, \texttt{sbm100}, and \texttt{port5} datasets. 
For the \texttt{port2} dataset, all the three methods failed to solve problem instances to optimality, but CPA${}_+$ found solutions of better quality than the other methods for all $\gamma \in \{1/\sqrt{N}, 10/\sqrt{N},100/\sqrt{N}\}$.  

The computation time of CPA${}_+$ tended to be shorter for smaller $\gamma$. 
For the \texttt{port5} dataset, CPA${}_+$ finished the computations in 98.2 s, 97.8 s, and 724.5 s for $\gamma\in \{1/\sqrt{N}, 10/\sqrt{N},100/\sqrt{N}\}$, respectively.
A similar tendency was also shown from the results for the \texttt{s\&p500} dataset, where the optimality gaps obtained by CPA${}_+$ were 4.9\%, 15.3\%, and 73.7\% for $\gamma\in \{1/\sqrt{N}, 10/\sqrt{N},100/\sqrt{N}\}$, respectively.

\begin{table}[p]
    \centering
    \caption{Numerical results with $k=10$ and $(\kappa_1,\kappa_2)=(1,4)$ for\\ $\gamma\in \{1/\sqrt{N}, 10/\sqrt{N}, 100/\sqrt{N}\}$}
         \label{tab:result_gamma}
     \footnotesize
      \renewcommand{\arraystretch}{0.90}
\begin{tabular}{lrrlrrrrr}
\toprule
Data &$N$& $\gamma$ 	&Method	 &Obj	 &Gap(\%)	&Time 	&\#Cuts 	&\#Nodes\\ \hline
\texttt{ind49}	&49	&1$/\sqrt{N}$	&SCIP	&3.415	&0.0	&1132.5	&---	&48\\
	&	&	&CPA	&3.415	&0.0	&434.6	&20	&20\\
	&	&	&CPA${}_+$	&3.415	&0.0	&\textbf{11.5}	&20	&20\\ \cline{3-9} 
	&	&10$/\sqrt{N}$	&SCIP	&3.034	&0.0	&663.2	&---	&24\\
	&	&	&CPA	&3.034	&0.0	&469.6	&22	&20\\
	&	&	&CPA${}_+$	&3.034	&0.0	&\textbf{12.0}	&22	&20\\ \cline{3-9} 
	&	&100$/\sqrt{N}$	&SCIP	&2.984	&0.0	&706.3	&---	&26\\
	&	&	&CPA	&2.984	&0.0	&481.0	&21	&36\\
	&	&	&CPA${}_+$	&2.984	&0.0	&\textbf{11.4}	&21	&39\\ \hline 
 \texttt{port2}	&89	&1$/\sqrt{N}$	&SCIP	&2.675	&32.0	&$>$3600	&---&$>$10\\
	&	&	&CPA	&2.402	&46.4	&$>$3600	&$>$24	&$>$291\\
	&	&	&CPA${}_+$	&\textbf{2.160}	&7.1	&$>$3600	&$>$6359	&$>$156591\\ \cline{3-9} 
	&	&10$/\sqrt{N}$	&SCIP	&1.865	&13.8	&$>$3600	&---&$>$17\\
	&	&	&CPA	&1.924	&160.2	&$>$3600	&$>$28	&$>$91\\
	&	&	&CPA${}_+$	&\textbf{1.726}	&55.2	&$>$3600	&$>$6380	&$>$110529\\ \cline{3-9} 
	&	&100$/\sqrt{N}$	&SCIP	&1.796	&12.1	&$>$3600	&---&$>$18\\
	&	&	&CPA	&1.847	&3528.8	&$>$3600	&$>$30	&$>$533\\
	&	&	&CPA${}_+$	&\textbf{1.675}	&506.7	&$>$3600	&$>$6496	&$>$60165\\ \hline 
 \texttt{sbm100}	&100	&1$/\sqrt{N}$	&SCIP	&6.323	&27.8	&$>$3600	&---&$>$5\\
	&	&	&CPA	&4.594	&0.0	&2186.2	&7	&1\\
	&	&	&CPA${}_+$	&4.594	&0.0	&\textbf{3.6}	&7	&1\\ \cline{3-9} 
	&	&10$/\sqrt{N}$	&SCIP	&3.970	&0.9	&$>$3600	&---&$>$9\\
	&	&	&CPA	&3.935	&0.0	&1491.5	&5	&1\\
	&	&	&CPA${}_+$	&3.935	&0.0	&\textbf{2.6}	&5	&1\\ \cline{3-9} 
	&	&100$/\sqrt{N}$	&SCIP	&3.806	&0.1	&$>$3600	&---&$>$9\\
	&	&	&CPA	&3.801	&0.0	&1941.4	&6	&1\\
	&	&	&CPA${}_+$	&3.801	&0.0	&\textbf{2.5}	&5	&1\\ \hline 
 \texttt{port5}	&225	&1$/\sqrt{N}$	&SCIP	&$\infty$	&100.0	&$>$3600	&---&$>$1\\
	    &   &   &CPA  & OM &--- &---&---&---\\
	    &	&	&CPA${}_+$	&3.380	&0.0	&\textbf{98.2}	&179	&2260\\ \cline{3-9} 
	&	&10$/\sqrt{N}$	&SCIP	&$\infty$	&100.0	&$>$3600	&---&$>$1\\
	    &   &   &CPA  & OM &--- &---&---&---\\
	    &	&	&CPA${}_+$	&2.687	&0.0	&\textbf{97.8}	&178	&1751\\ \cline{3-9} 
	&	&100$/\sqrt{N}$	&SCIP	&$\infty$	&100.0	&$>$3600	&---&$>$1\\
	    &   &   &CPA  & OM &--- &---&---&---\\
	    &	&	&CPA${}_+$	&2.611	&0.0	&\textbf{724.5}	&1318	&21438\\ \hline 
 \texttt{s\&p500}	&468	&1$/\sqrt{N}$	&SCIP	&$\infty$	&---	&$>$3600	&---&$>$1\\
		    &   &   &CPA  & OM &--- &---&---&---\\
		    &	&	&CPA${}_+$	&\textbf{1.875}	&4.9	&$>$3600	&$>$4949	&$>$431321\\ \cline{3-9} 
	&	&10$/\sqrt{N}$	&SCIP	&$\infty$	&---	&$>$3600	&---&$>$1\\
	    &   &   &CPA  & OM &--- &---&---&---\\
	    &	&	&CPA${}_+$	&\textbf{0.888}	&15.3	&$>$3600	&$>$4962	&$>$398075\\ \cline{3-9} 
	&	&100$/\sqrt{N}$	&SCIP	&$\infty$	&---	&$>$3600	&---&$>$1\\
		    &   &   &CPA  & OM &--- &---&---&---\\
		    &	&	&CPA${}_+$	&\textbf{0.781}	&73.7	&$>$3600	&$>$5736	&$>$145328\\   
	\bottomrule
\end{tabular}
\end{table}

\section{Results of optimal investment weights}\label{apd:d}
\Cref{fig:stacked_barplot_dr_sparse,fig:stacked_barplot_dr_dense,fig:stacked_barplot_ro_sparse,fig:stacked_barplot_ro_dense,fig:stacked_barplot_ind49_dr_sparse,fig:stacked_barplot_ind49_dr_dense,fig:stacked_barplot_ind49_ro_sparse,fig:stacked_barplot_ind49_ro_dense,fig:stacked_barplot_sbm100_dr_sparse,fig:stacked_barplot_sbm100_dr_dense,fig:stacked_barplot_sbm100_ro_sparse,fig:stacked_barplot_sbm100_ro_dense} show the stacked bar plots of optimal investment weights given by the DR and RO models with $k\in \{5,N\}$  for the three datasets (\texttt{nikkei225}, \texttt{ind49}, and \texttt{sbm100}) in \Cref{sec:exp_oos}. 
From \Cref{fig:stacked_barplot_dr_sparse,fig:stacked_barplot_ind49_dr_sparse,fig:stacked_barplot_sbm100_dr_sparse}, we see that the optimal portfolios obtained by the DR model with $k=5$ are moderately diversified for each dataset.

\pgfplotsset{/pgfplots/bar cycle list/.style={/pgfplots/cycle list name={customlist}}}

\begin{figure}[p]
    \centering
\includegraphics[page=1]{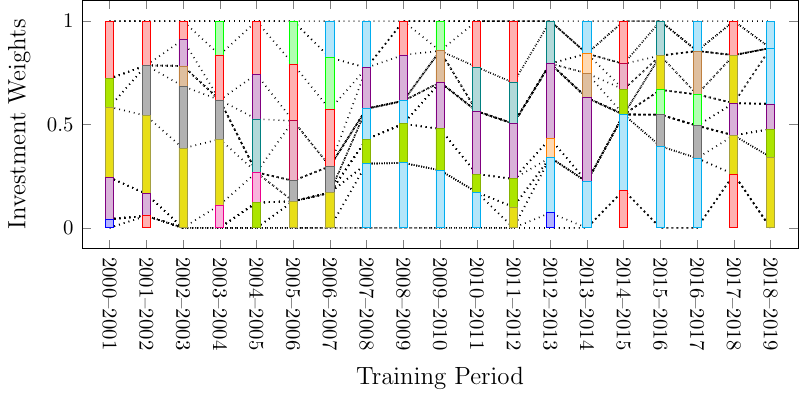} 
\caption{Optimal investment weights provided by the DR model with $k=5$ for the \texttt{nikkei225} dataset}
\label{fig:stacked_barplot_dr_sparse}
\end{figure}

\begin{figure}
    \centering
\includegraphics[page=2]{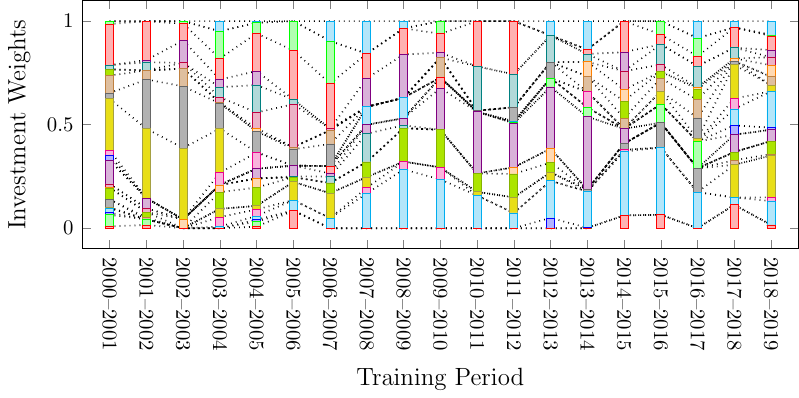} 
\caption{Optimal investment weights provided by the DR model with $k=N$ for the \texttt{nikkei225} dataset}
\label{fig:stacked_barplot_dr_dense}
\end{figure}

\begin{figure}
    \centering
\includegraphics[page=3]{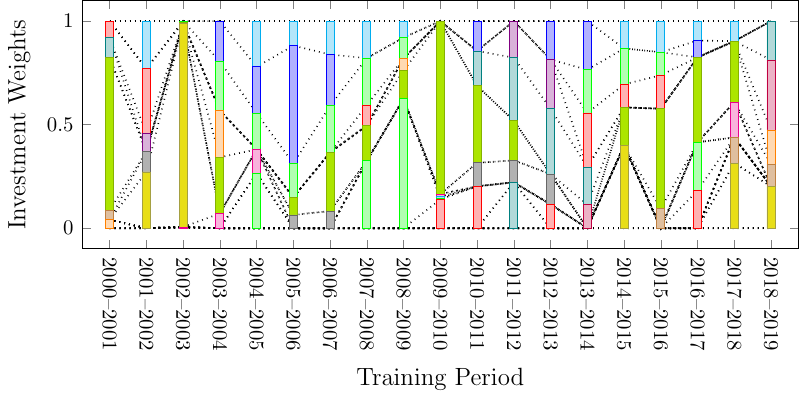} 
\caption{Optimal investment weights provided by the RO model with $k=5$ for the \texttt{nikkei225} dataset}
\label{fig:stacked_barplot_ro_sparse}
\end{figure}

\begin{figure}
    \centering
\includegraphics[page=4]{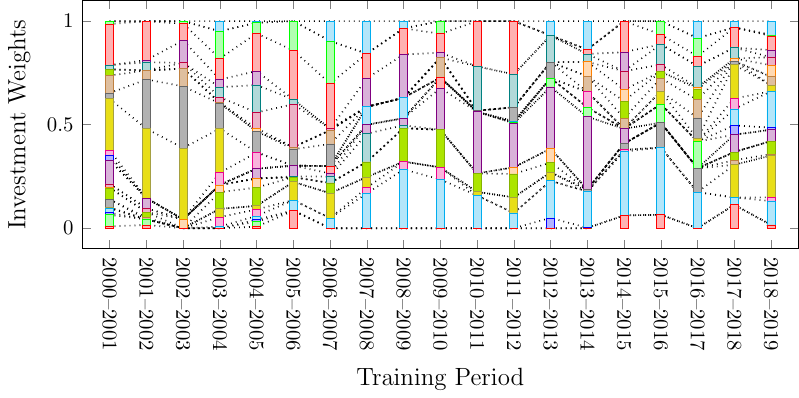} 
\caption{Optimal investment weights provided by the RO model with $k=N$ for the  \texttt{nikkei225} dataset}
\label{fig:stacked_barplot_ro_dense}
\end{figure}

\begin{figure}[p]
    \centering
\includegraphics[page=1]{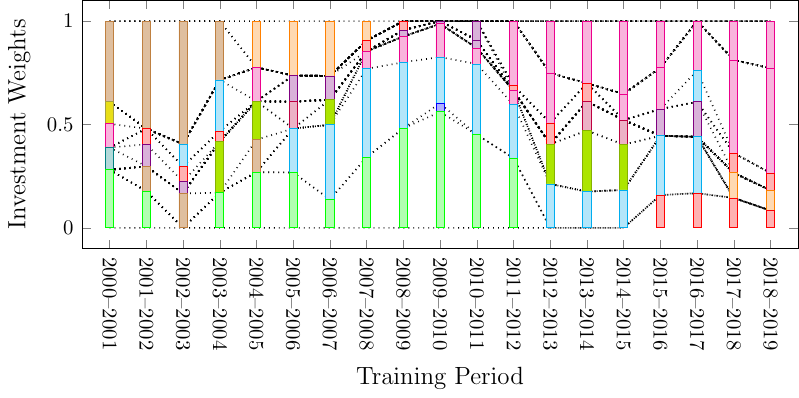} 
\caption{Optimal investment weights provided by the DR model with $k=5$ for the \texttt{ind49} dataset}
\label{fig:stacked_barplot_ind49_dr_sparse}
\end{figure}

\begin{figure}
    \centering
\includegraphics[page=2]{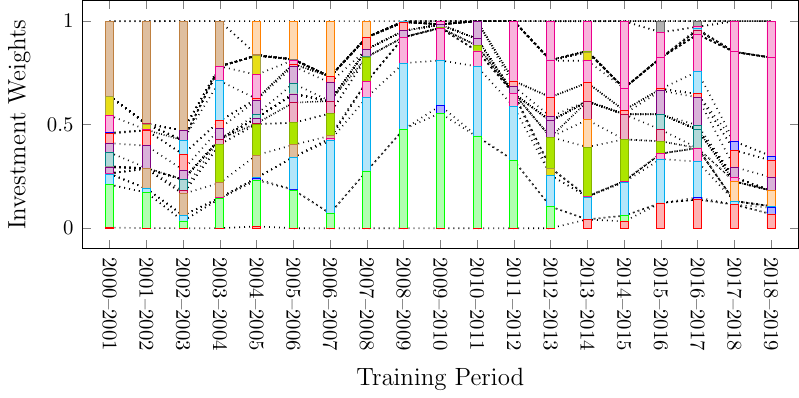} 
\caption{Optimal investment weights provided by the DR model with $k=N$ for the \texttt{ind49} dataset}
\label{fig:stacked_barplot_ind49_dr_dense}
\end{figure}

\begin{figure}
    \centering
\includegraphics[page=3]{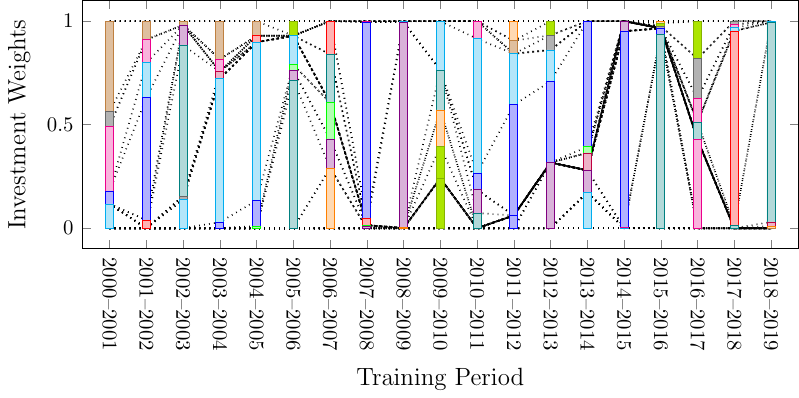} 
\caption{Optimal investment weights provided by the RO model with $k=5$ for the \texttt{ind49} dataset}
\label{fig:stacked_barplot_ind49_ro_sparse}
\end{figure}

\begin{figure}
    \centering
\includegraphics[page=4]{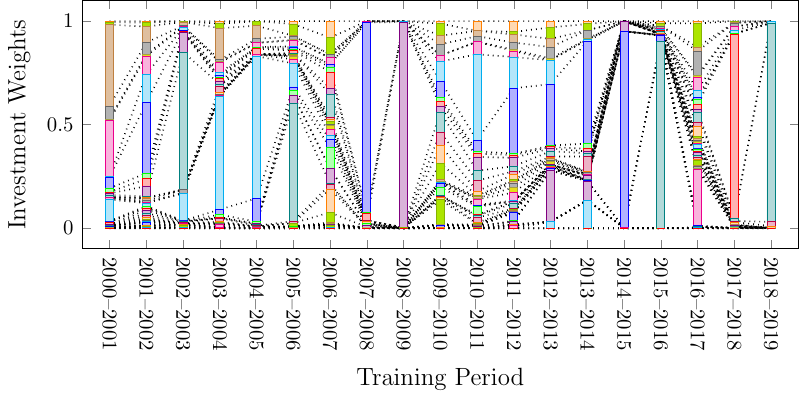} 
\caption{Optimal investment weights provided by the RO model with $k=N$ for the  \texttt{ind49} dataset}
\label{fig:stacked_barplot_ind49_ro_dense}
\end{figure}

\begin{figure}[p]
    \centering
\includegraphics[page=1]{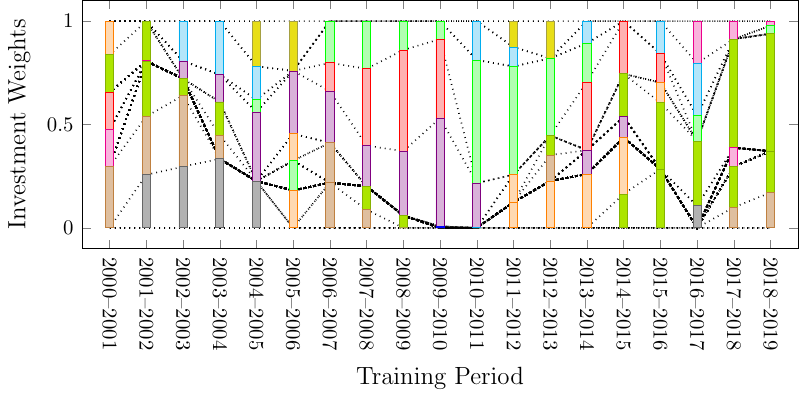} 
\caption{Optimal investment weights provided by the DR model with $k=5$ for the \texttt{sbm100} dataset}
\label{fig:stacked_barplot_sbm100_dr_sparse}
\end{figure}

\begin{figure}
    \centering
\includegraphics[page=2]{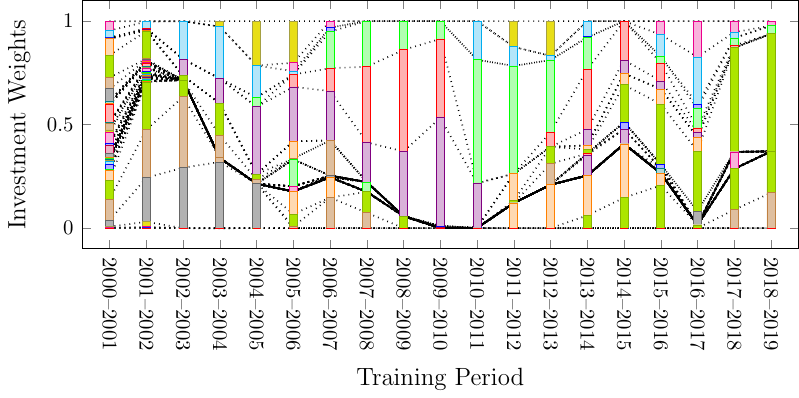} 
\caption{Optimal investment weights provided by the DR model with $k=N$ for the \texttt{sbm100} dataset}
\label{fig:stacked_barplot_sbm100_dr_dense}
\end{figure}

\begin{figure}
    \centering
\includegraphics[page=3]{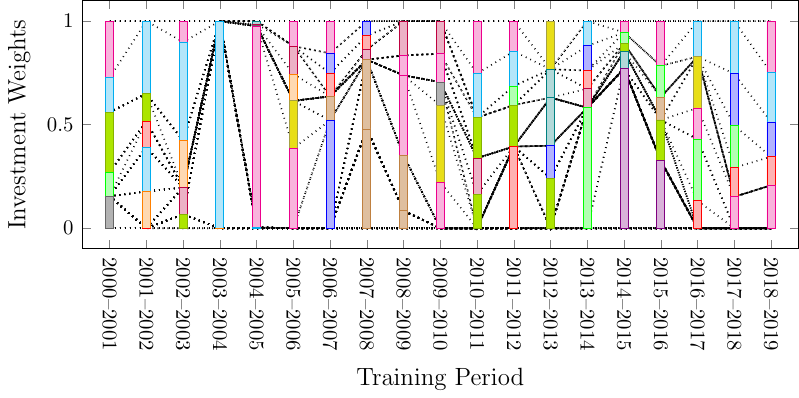} 
\caption{Optimal investment weights provided by the RO model with $k=5$ for the \texttt{sbm100} dataset}
\label{fig:stacked_barplot_sbm100_ro_sparse}
\end{figure}

\begin{figure}
    \centering
\includegraphics[page=4]{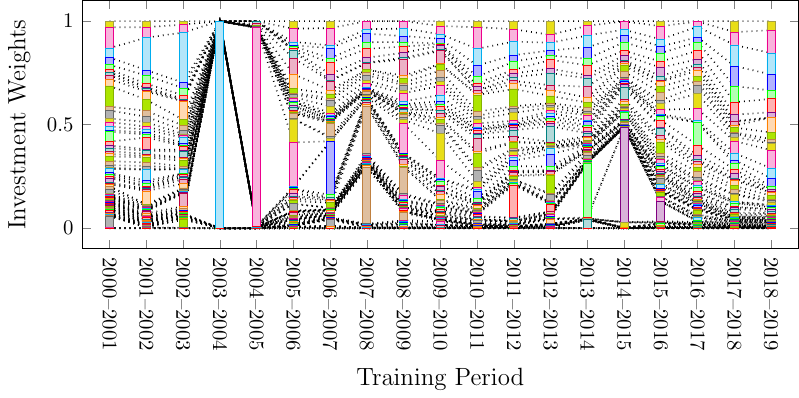} 
\caption{Optimal investment weights provided by the RO model with $k=N$ for the  \texttt{sbm100} dataset}
\label{fig:stacked_barplot_sbm100_ro_dense}
\end{figure}

\end{document}